\documentclass[11pt]{amsart}
\usepackage{amsmath, amssymb, array}
\usepackage[top=30truemm,bottom=30truemm,left=30truemm,right=30truemm]{geometry} 
\usepackage{mathtools}
\usepackage[all]{xy}
\usepackage{ascmac}
\usepackage[dvipdfmx]{graphicx}
\usepackage{color}
\usepackage{amscd}
\usepackage{fancyhdr}
\usepackage[bookmarks=false]{hyperref} 
\usepackage{amsrefs}
\usepackage{cleveref}
\usepackage{chngcntr}
\usepackage{apptools}
\usepackage{comment}
\usepackage{listings}
\usepackage{booktabs}
\usepackage{multirow}
\usepackage{url}

\renewcommand{\Re}{\mathrm{Re}} 

\DeclareMathOperator{\PTE}{PTE} 
\DeclareMathOperator{\im}{Im} 
\DeclareMathOperator{\ord}{ord} 

\DeclareMathOperator{\GL}{GL}

\title{Ellipsoidal designs and the Prouhet--Tarry--Escott problem}

\author{Hideki Matsumura$^{*}$}
\email{hmatsumura@tmu.ac.jp}
\address{${}^*$
Graduate School of Science, Tokyo Metropolitan University, 1-1 Minami-Osawa, Hachioji-shi, Tokyo 192-0397, Japan}
\author{Masanori Sawa$^{**}$}
\email{sawa@people.kobe-u.ac.jp}
 \address{${}^{**}$Graduate School of System Informatics, Kobe University, 1-1 Rokkodai, Nada, Kobe 657-8501, Japan}

\thanks{This research is supported by KAKENHI 24KJ0183, KAKENHI 25K17234 
and KAKENHI 50508182 of the Japan Society for the Promotion of Science (JSPS), and by the Early Support Program for Grant-in-Aid for Scientific Research of Kobe University.} 

\subjclass[2020]{primary 05E99, 11D72; secondary 05B30, 14G05} 
 \keywords{Rational ellipsoidal designs, tight design, Prouhet--Tarry--Escott problem, Borwein solution, Alpers--Tijdeman solution, linearity and symmetry}
\date{\today}

\theoremstyle{plain}
 \newtheorem{theorem}{Theorem}[section] 
  \crefname{theorem}{Theorem}{Theorems}
 \newtheorem{proposition}[theorem]{Proposition}
 \crefname{proposition}{Proposition}{Propositions}
 \newtheorem{lemma}[theorem]{Lemma}
 \crefname{lemma}{Lemma}{Lemmas}
 
  \crefname{corollary}{Corollary}{Corollaries}
 
   \crefname{conjecture}{Conjecture}{Conjectures}
 
 \crefname{question}{Question}{Questions}
 \newtheorem{problem}[theorem]{Problem}
   \crefname{problem}{Problem}{Problems}
 
    \crefname{notation}{Notation}{Notations}
\crefname{table}{Table}{Tables}

\theoremstyle{definition} 
 \newtheorem{definition}[theorem]{Definition}
  \crefname{definition}{Definition}{Definitions}
 \newtheorem{example}[theorem]{Example}
   \crefname{example}{Example}{Examples}
 \newtheorem{remark}[theorem]{Remark}
   \crefname{remark}{Remark}{Remarks}
    
   \crefname{claim}{Claim}{Claims}

\begin{document}

\maketitle

\tableofcontents

\begin{abstract}
The notion of ellipsoidal design was first introduced by Pandey (2022) as a full generalization of spherical designs on the unit circle $S^1$.
In this paper, we elucidate the advantages of examining the connections between ellipsoidal design and the two-dimensional Prouhet--Tarry--Escott problem, say $\PTE_2$, 
originally introduced by Alpers and Tijdeman (2007) as a natural generalization of the classical one-dimensional PTE problem ($\PTE_1$). 
We first provide a combinatorial criterion for the construction of solutions of $\PTE_2$ from a pair of ellipsoidal designs.
We also give an arithmetic proof of the Stroud-type bound for ellipsoidal designs, and then establish a classification theorem for designs with equality.
Such a classification result is closely related to an open question on the existence of rational spherical $4$-designs on $S^1$, discussed in Cui, Xia and Xiang (2019).
As far as the authors know, a solution found by Alpers and Tijdeman is the first and the only known parametric ideal solution of degree $5$ for $\PTE_2$.
Moreover, as one of our main theorems, we prove that the Alpers--Tijdeman solution is equivalent to a certain two-dimensional extension of the famous 
Borwein solution for $\PTE_1$.
As a by-product of this theorem, we discover a family of ellipsoidal $5$-designs among the Alpers--Tijdeman solution. 
\end{abstract}

\section{Introduction} \label{sect:Intro}

One of the most significant results in the interaction between geometric designs and integral lattices, is the Venkov theorem (see \cite{Venkov1984}) that derives the existence of spherical designs from shells (or layers) of integral lattices under reasonable assumptions.
Since the pioneering work by Venkov, there have been numerous publications on the construction of spherical designs via shells of integral lattices; see for example \cite{HPV2007,Pache2005}.
Such lattice-based constructions are also of great importance for coding theorists as a spherical analogue of the Assmus--Mattson theorem that provides combinatorial designs from shells of linear codes.
For a good introduction to the relationships among designs, codes and lattices, we refer the reader to Conway and Sloane \cite{CS1999}, Bannai and Bannai~\cite{BB2009}, 
and Bannai, Bannai, Ito and Tanaka \cite{BBIT_book} and references therein.

de la Harpe, Pache and Venkov~\cite{HPV2007} established that any shell of the $E_8$-root lattice never yields a spherical $8$-design if and only if the Ramanujan tau function $\tau(m)$ never vanishes at all positive integers $m$, which directly relates spherical designs to the famous Lehmer conjecture.
Afterwards, many researchers worked on discussing variations of the result by de la Harpe et al., as exemplified by~\cite{BKST2006},~\cite{BM2010} and~\cite{HNT2023}.
A particularly  influential work is Bannai and Miezaki~\cite{BM2010}, which states that any shell of the $\mathbb{Z}^2$-lattice is never a spherical $4$-design; for a generalization of this result, we refer the reader to Miezaki~\cite{Miezaki}.

Inspired by the aforementioned works, Pandey~\cite{Pandey2022} introduced the notion of ellipsoidal design.
A finite set $X$ on the ellipse
\begin{align*}
C_D(r):= \begin{cases}
\{(x,y) \in \mathbb{R}^2 \mid x^2+Dy^2=r \} & (D \equiv 1, \; 2 \pmod{4}),\\
\left\{(x,y) \in \mathbb{R}^2 \mid x^2+xy+\frac{1+D}{4}y^2=r \right\} & (D \equiv 3 \pmod{4})
\end{cases}
\end{align*}
is called an {\it ellipsoidal $n$-design} if
\begin{align*}
\frac{1}{\# X} \sum_{(x,y) \in X} f(x,y) &= \begin{cases}
\frac{1}{2 \pi \sqrt{D}} \int_{C_D(r)} f(x,y) d \tau(x,y) & (D \equiv 1, \; 2 \pmod{4}), \\
\frac{\sqrt{D}}{\pi} \int_{C_D(r)} f(x,y) d \tau(x,y) & (D \equiv 3 \pmod{4})
\end{cases}
\end{align*}
for all polynomials $f$ of degree at most $n$, where $d\tau$ is the invariant measure on $C_D(r)$.
In particular when $D = 1$, the point set $X$ is called a {\it spherical $n$-design}
on the unit circle $S^1$.
Pandey established a full generalization of the results by Bannai and Miezaki, and by Miezaki.

Pandey's work is not only crucial in the connections among (geometric) designs, codes and lattices, but also related to a classical Diophantine problem called the Prouhet--Tarry--Escott (PTE) problem, as will be clear in this paper.
Given $m, n \in \mathbb{N}$, {\it the PTE problem $(\PTE_1)$ of degree $n$ and size $m$},
asks whether there exists a disjoint pair of multisets
\[
\{a_1, \ldots, a_m\}, \{b_1, \ldots , b_m\} \subset \mathbb{Z}
\]
such that
\begin{align*}
\sum_{i=1}^m a_i^k=\sum_{i=1}^m b_i^k \quad (1 \leq k \leq n).
\end{align*}
For an early history concerning the PTE problem, we refer the reader to Dickson's book~\cite[\S~24]{Dickson}.
Borwein~\cite[Chapter 11]{Borwein2002} presents a brief survey on this subject, covering various {\it ideal solutions}, 
i.e., solutions with $m=n+1$. Existence, as well as constructions of ideal solutions, are among the central core topics of the classical PTE problem.
Among many others, a particularly important solution due to Borwein~\cite[p.\ 88]{Borwein2002} is given by
\begin{multline} \label{Borwein}
 [\pm (2m+2n), \pm (-nm-n-m+3), \pm (nm-n-m-3)] \\
=_5 [\pm (2n-2m), \pm (nm-n+m+3), \pm (-nm-n+m-3)].
\end{multline}
This can be cyclically extended to a family of solutions, which we call the two-dimensional Borwein solution (\cref{2DBorwein} of this paper)
 for the two-dimensional PTE problem originally introduced by Alpers--Tijdeman~\cite{Alpers-Tijdeman}.

{\it The two-dimensional PTE problem $(\PTE_2)$ of degree $n$ and size $m$} asks whether there exists a disjoint pair of multisets
\[
\{(a_{11}, a_{12}), \ldots, (a_{m1}, a_{m2})\}, \{(b_{11}, b_{12}), \ldots, (b_{m1}, b_{m2})\} \subset \mathbb{Z}^2
\]
such that
\[
\sum_{i=1}^m a_{i1}^{k_1} a_{i2}^{k_2}=\sum_{i=1}^m b_{i1}^{k_1} b_{i2}^{k_2} 
\quad (1 \leq k_1+k_2 \leq n).
\]
Alpers and Tijdeman \cite[\S 2.4]{Alpers-Tijdeman } found the ideal solution 
\begin{multline} \label{ATsol}
[(0,0),(2a + b,b),(3a + b,3a + 3b),(2a,6a + 4b),(-b,6a + 3b),(-a-b,3a + b)]\\
=_5 [(2a,0),(3a + b,3a + b),(2a + b,6a + 3b),(0,6a + 4b),(-a-b,3a + 3b),(-b,b)].
\end{multline} 
This has been the only known parametric ideal solution of degree $5$ for $\PTE_2$, which, to our surprise, includes a family of ellipsoidal $5$-designs 
(see Theorem~\ref{2BvsAT} of this paper).  

The classical one-dimensional PTE problem
is closely tied with other mathematical subjects, including the study of zeros of chromatic polynomials of a certain class of
finite graphs~\cite{HL},
the decomposability of a certain polynomial that can be expressed as a product of linear terms over $\mathbb{Q}$~\cite{HPT} and so on.
Also, in recent years, the applied aspects of $\PTE_1$ have been studied, as exemplified by an influential work by Costello et al.~\cite{CMN} that provides a sieving algorithm for finding pairs of consecutive smooth integers for compact isogeny-based post-quantum protocols.
Whereas, only a few publications, except for Alpers--Tijdeman~\cite{Alpers-Tijdeman } and Ghiglione~\cite{Ghiglione}, have been devoted to the connections to other mathematical subjects, as well as the applications of $\PTE_2$.

Thus in this paper, we create novel connections between $\PTE_2$ and ellipsoidal designs.
Since $S^1$ and $C_D(r)$ are homeomorphic, there may not seem to be much difference between spherical design and ellipsoidal design.
However, by considering rational points, we can find an essential gap between the two notions.
For example, we prove (\cref{Dstr}) that there exist infinitely many {\it rational} ellipsoidal $5$-designs for $D=3$, 
i.e., designs for which the components of each point are all rational numbers, whereas the existence of rational $4$-designs over $S^1$ is still an open problem; see \cref{rat} and \cite[\S 1.3]{CXX} for details.
 Incidentally, the existence of rational spherical designs can also be reduced to finding rational designs for Gegenbauer measure
$(1-x^2)^{\lambda-1/2} dx$ ($\lambda >-1/2$), and is a significant research subject in algebro-combinatorial
design theory; for example, see \cite[p.\ 208, Problem 2]{BBIT_book}. 

This paper is organized as follows:
\S2 is a brief introduction to basic facts on two-dimensional ellipsoidal harmonics that are used throughout this paper.
\S3 through \S5 are the main body of this paper.
In \S3 we prove Sobolev's theorem for ellipsoidal designs (\cref{ellSobolev}), which provides an equivalent form of the definition of ellipsoidal designs invariant under a finite subgroup of the 
^^ orthogonal group' $\mathrm{O}(Q)$ where $Q$ is a positive definite matrix in $M_2(\mathbb{Q})$; see (\ref{Q}) and (\ref{OQ}) for the definition of $\mathrm{O}(Q)$. 
In \S4, we give an arithmetic proof of the Stroud-type bound for ellipsoidal designs by combining 
the invariance of the measure $d\tau$ and a certain fundamental inequality for $\PTE_2$ (\cref{ESB}).
Moreover, we establish a classification theorem for tight designs, and then prove that there exists a rational tight ellipsoidal design only if $D=1$ or $3$ (\cref{tight}).
In \S5, we first provide a combinatorial criterion for the construction of solutions of $\PTE_2$ from a pair of ellipsoidal designs (\cref{DtoPTE}), and then obtain a family of ideal solutions of degree $5$ (\cref{Dstr}).
 Moreover, we establish a two-dimensional extension of Borwein's solution (\cref{2DBorwein}).
As one of the main results, we prove that the two-dimensional Borwein solution (\ref{2dimBorwein}) is equivalent to the Alpers--Tijdeman solution (\ref{ATsol}) over $\mathbb{Q}$, 
and in particular discover an infinite family of ellipsoidal $5$-designs among the Alpers--Tijdeman solution (\cref{2BvsAT}).
We emphasize that \cref{2BvsAT} provides not only an alternative construction of the Alpers--Tijdeman solution, but also a design-theoretic insight of the original geometric work by Alpers and Tijdeman. 
\S6 is the conclusion, where further remarks and open questions will be made.

It seems that recent developments of the PTE problem are not fully recognized by researchers who study spherical designs or ellipsoidal designs in combinatorics and algebraic number theory, and conversely, 
the geometric design theory is not known among people who work with the PTE problem and related topics in additive number theory.
It is our hope that the present paper could be of broad interest and serve as a new research platform among people working in those areas.

\section{Ellipsoidal harmonics} \label{sect:ellipsoidal}

This section is basically along the lines of \cite{Pandey2022}. 
Most of the results on two-dimensional ellipsoidal harmonics, given in this section, seem to be known somewhere else, but the authors could not trace it back.
This section thus makes a brief introduction to this topics.

For square-free $D \in \mathbb{Z}_{\geq 1}$ and $r \in \mathbb{R}_{>0}$, we define the {\it norm $r$ ellipses} by
\begin{align*} \label{ell}
C_D(r) = \begin{cases}
\{(x,y) \in \mathbb{R}^2 \mid x^2+Dy^2=r \} & (D \equiv 1, \; 2 \pmod{4}),\\
\left\{(x,y) \in \mathbb{R}^2 \mid x^2+xy+\frac{1+D}{4}y^2=r \right\} & (D \equiv 3 \pmod{4}).
\end{cases}
\end{align*}

\begin{remark} [{Cf.\ \cite[p.\ 1246 Remark]{Pandey2022}}] 
Let $K=\mathbb{Q}(\sqrt{-D})$ and $\mathcal{O}_K$ be the ring of integers of $K$.
Then the conditions for $C_D(r)$ are the norms of elements in $\mathcal{O}_K$.
Indeed, if $D \equiv 1, \; 2 \pmod{4}$, then the norm of $x+y \sqrt{-D} \in \mathcal{O}_K=\mathbb{Z}[\sqrt{-D}]$ is
\[
 N(x+y\sqrt{-D})=(x+y\sqrt{-D})(x-y\sqrt{-D})=x^2+Dy^2.
\]
If $D \equiv 3 \pmod{4}$, then the norm of $x+y (1+\sqrt{-D})/2 \in \mathcal{O}_K=\mathbb{Z}[(1+\sqrt{-D})/2]$ is 
\[
 N \left(x+y\frac{1+\sqrt{-D}}{2} \right)=\left(x+y\frac{1+\sqrt{-D}}{2} \right) \left(x+y \frac{1-\sqrt{-D}}{2} \right)=x^2+xy+\frac{1+D}{4} y^2.
\]
\end{remark}
 
Let $s$ be the arc-length parameter on $C_D(r)$.
We define the measure
 $d \tau$ on $C_D(r)$ by
\begin{align} \label{tau}
d\tau(x,y) =
\begin{cases}
 \displaystyle\frac{ds}{\sqrt{x^2/D^2+y^2}} & (D \equiv 1, \; 2 \pmod{4}),\\
  \displaystyle\frac{ds}{\sqrt{20x^2+(D^2+2D+5)y^2+(20+4D)xy}} & (D \equiv 3 \pmod{4}).
\end{cases}
\end{align}

\begin{remark} [{Cf.\ Proof of \cite[eqs. (1.4), (1.5), Theorem 1.1]{Pandey2022}}] \label{measure}
We parametrize $C_D(r)$ by 
\[(x(\theta),y(\theta))=
\begin{cases}
(\sqrt{r} \cos \theta, \sqrt{r}\sin \theta/\sqrt{D}) & (D \equiv 1, \; 2 \pmod{4}),\\
(\sqrt{r}(\cos \theta-\sin \theta/\sqrt{D}),2\sqrt{r} \sin \theta/\sqrt{D}) & (D \equiv 3 \pmod{4}).
\end{cases} \]
Since 
\[x'(\theta)^2+y'(\theta)^2=
\begin{cases}
r((-\sin \theta)^2+(\cos \theta/\sqrt{D})^2)=r(\sin^2 \theta+\cos^2 \theta/D ) & (D \equiv 1, \; 2 \pmod{4}),\\
r(\sin^2 \theta+5 \cos^2 \theta/D+ 2 \sin \theta \cos \theta/ \sqrt{D}) & (D \equiv 3 \pmod{4}),
\end{cases}\]
 we obtain
\begin{align*}
\begin{cases} 
\frac{1}{2 \pi \sqrt{D}} \int_{C_D(r)} f(x,y) d \tau(x,y)
=\frac{1}{2 \pi} \int_{0}^{2\pi} f(x(\theta),y(\theta)) d \theta & (D \equiv 1, \; 2 \pmod{4}),\\
\frac{\sqrt{D}}{\pi} \int_{C_D(r)} f(x,y) d \tau(x,y)
=\frac{1}{2 \pi} \int_{0}^{2\pi} f(x(\theta),y(\theta)) d \theta  & (D \equiv 3 \pmod{4}).
\end{cases}
\end{align*}
Note that the above parametrizations can be obtained from
\[
P^{-1}
\begin{pmatrix}
\sqrt{r} \cos \theta\\
\sqrt{r} \sin \theta
\end{pmatrix},
\]
where $P$ is the matrix in the proof of \cref{Lap}.
\end{remark}

Pandey \cite{Pandey2022} generalized spherical designs to ellipsoidal designs.
\begin{definition} [{Ellipsoidal $n$-design}] 
Let $X \subset C_D(r)$ be a finite non-empty subset and $d\tau$ be the measure given by (\ref{tau}). 
$X$ is an {\it ellipsoidal $n$-design} if
\begin{align} \label{EtD}
\frac{1}{\# X} \sum_{(x,y) \in X} f(x,y)
&= \begin{cases}
\frac{1}{2 \pi \sqrt{D}} \int_{C_D(r)} f(x,y) d \tau(x,y) & (D \equiv 1, \; 2 \pmod{4}), \\
\frac{\sqrt{D}}{\pi} \int_{C_D(r)} f(x,y) d \tau(x,y) & (D \equiv 3  \pmod{4})
\end{cases}
\end{align}
for all polynomials $f(x,y)$ of degree $\leq n$.
In particular when $D=1$, the point set $X$ is called a {\it spherical $n$-design} on $S^1$.
\end{definition}

\begin{definition} [{Ellipsoidal homogeneous space}] 
\label{HDjR}
We define a vector space of bivariate homogeneous polynomials by
\begin{align*}
H_{D,j}^{\mathbb{R}}[x,y] =
\begin{cases}
\langle \Re (x+y\sqrt{-D})^j , \; \im (x+y\sqrt{-D})^j \rangle_{\mathbb{R}} & (D \equiv 1, \; 2 \pmod{4}),\\
\left\langle \Re \left(x+y \frac{1+\sqrt{-D}}{2} \right)^j ,\; \im \left(x+y \frac{1+\sqrt{-D}}{2} \right)^j \right\rangle_{\mathbb{R}}& (D \equiv 3 \pmod{4}).
\end{cases}
\end{align*}
\end{definition}

For a positive definite matrix $Q={}^t Q \in M_2(\mathbb{Q})$, 
we define the {\it Laplacian} $\Delta_Q$ of $Q$ by
\[
\Delta_Q = \sum_{i,j=1}^2 q^{i,j} \frac{\partial^2}{\partial x_i \partial x_j},
\]
where
$Q^{-1}=(q^{i,j})$.
In what follows, let 
\begin{align}
Q =
\begin{cases} \label{Q}
\begin{pmatrix}
1 & 0\\
0 & D
\end{pmatrix} & (D \equiv 1, 2 \pmod{4}),\\
\begin{pmatrix}
1 & \frac{1}{2}\\
\frac{1}{2} & \frac{1+D}{4}
\end{pmatrix} & (D \equiv 3 \pmod{4}).
\end{cases}
\end{align}

The following result, not mentioned in Pandey \cite{Pandey2022}, develops a relation between $H_{D,j}^{\mathbb{R}}[x,y]$ and $\Delta_Q$.
\begin{proposition} \label{Lap}
$H_{D,j}^{\mathbb{R}}[x,y]$ coincides with the space of $Q$-harmonic polynomials of degree $j$ defined by
\[\{f(x,y) \in \mathbb{R}[x,y] \mid \deg f=j \; \mbox{and} \; \Delta_Q f=0 \}. \]
\end{proposition}
\begin{proof} 
We reduce to the case $D=1$, which is classically well-known (see \cite{EMOT} for example).
Let
\begin{align*}
P =
\begin{cases}
\begin{pmatrix}
1 & 0\\
0 & \sqrt{D}
\end{pmatrix} & (D \equiv 1,\; 2 \pmod{4}),\\
\begin{pmatrix}
1 & 1/2\\
0 & \sqrt{D}/2
\end{pmatrix} & (D \equiv 3 \pmod{4}).
\end{cases}
\end{align*}
We do a change of the variable given by 
\begin{align*}
\begin{pmatrix}
x'\\
y'
\end{pmatrix}=
P\begin{pmatrix}
x\\
y
\end{pmatrix}
\end{align*}
and let 
\[
\Delta = \frac{\partial^2}{\partial x'^2}+\frac{\partial^2}{\partial y'^2}. 
\]
Then it follows by the chain rule that $\Delta'=\Delta_Q$. This completes the proof.
\end{proof}

As in the case of spherical designs, the integrations in (\ref{EtD}) are invariant under
the action of the orthogonal group.
For a positive definite matrix $Q$ given in (\ref{Q}), we define
the {\it orthogonal group} $\mathrm{O}(Q)$ by
\begin{align} \label{OQ}
\mathrm{O}(Q) = \{A \in \GL_2(\mathbb{R}) \mid  {}^tA QA=Q\}. 
\end{align}

\begin{proposition} \label{OQ-inv}
Let $0 \leq \phi< 2 \pi$. Then $\mathrm{O}(Q)$ is generated by
\begin{align*}
A:=\begin{cases}
\begin{pmatrix}
1 & 0\\
0 & -1
\end{pmatrix},\\ 
\begin{pmatrix}
1 & 1\\
0 & -1
\end{pmatrix}, 
\end{cases}
&& A_{\phi}:=
\begin{cases}
\begin{pmatrix}
\cos \phi & -\sqrt{D} \sin \phi\\
\frac{\sin \phi}{\sqrt{D}}& \cos \phi
\end{pmatrix} & (D \equiv 1, 2 \pmod{4}),\\
\begin{pmatrix}
\cos \phi-\frac{\sin \phi}{\sqrt{D}} & -\frac{(D+1)}{2\sqrt{D}} \sin \phi\\
\frac{2\sin \phi}{\sqrt{D}}& \cos \phi+\frac{\sin \phi}{\sqrt{D}}
\end{pmatrix} & (D \equiv 3 \pmod{4}).
\end{cases}
\end{align*}
Moreover, $d\tau$ is $\mathrm{O}(Q)$-invariant.
\end{proposition}

\begin{proof} 
Let $P$ be as in the proof of \cref{Lap}.
Suppose that $B \in \mathrm{O}(Q)$.
Since ${}^tB QB=Q$ and $Q={}^tP P$, we have $PBP^{-1} \in \mathrm{O}(2)$.
Note that $\mathrm{O}(Q)$ is generated by
\[\begin{pmatrix}
1 & 0\\
0 & -1
\end{pmatrix}, \;
\begin{pmatrix}
\cos \phi & - \sin \phi\\
\sin \phi& \cos \phi
\end{pmatrix}.\]
The invariance of $d \tau$ can be checked by direct calculations as in \cref{measure}. 
\end{proof}

The following provides an equivalent form of the definition of ellipsoidal $n$-designs: 

\begin{theorem} [{\cite[Theorem 1.1]{Pandey2022}}] \label{HDj} 
Let $X$ be a finite non-empty subset of $C_D(r)$. Then the following are equivalent: 
\begin{enumerate}
\item[(i)] $X$ is an ellipsoidal $n$-design.
\item[(ii)] It holds that
\begin{align} \label{EnD}
\sum_{(x,y) \in X} f(x,y) = 0
\end{align}
for all $f(x,y) \in H_{D,j}^{\mathbb{R}}[x,y]$ with $0 <j \leq n$.
\end{enumerate}
\end{theorem}

As in the classical spherical harmonics (\cite[Theorem 1.1.3 (2)]{BBbook}), the space of homogeneous polynomials can be represented by harmonic polynomials.

\begin{theorem} [{\cite[Lemma 2.1, Proof of Theorem 1.1]{Pandey2022}}] \label{dec}
Let $D \geq 1$ be a square-free integer and $P_j^{\mathbb{R}}[x,y]$ be the set of homogeneous polynomials of degree $j$.
Then it holds that
\begin{align*}
P_j^{\mathbb{R}}[x,y] =
\begin{cases}
\displaystyle\bigoplus_{j=0}^{\lfloor \frac{k}{2} \rfloor} (x^2+Dy^2)^k H_{D,j-2k}^{\mathbb{R}}[x,y] & (D \equiv 1, \; 2 \pmod{4}),\\
\displaystyle\bigoplus_{j=0}^{\lfloor \frac{k}{2} \rfloor} \left(x^2+xy+\frac{1+D}{4}y^2 \right)^k H_{D,j-2k}^{\mathbb{R}}[x,y] & (D \equiv 3 \pmod{4}).
\end{cases}
\end{align*}
In particular, the space of polynomials when restricted to $C_D(r)$ of degree $\leq j$ is a direct sum of $H_{D,k}^{\mathbb{R}}[x,y]$ with $0<k \leq j$.
\end{theorem}
\begin{remark}
In \cite[Lemma 2.1 (1)]{Pandey2022}, 
$D \equiv 3 \pmod {4}$ should be $D \equiv 1, \; 2 \pmod{4}$.
In the proof for $D \equiv 3 \pmod {4}$, $(x',y')=(x+y/2, 2y/\sqrt{D})$ should be $(x',y')=(x+y/2, \sqrt{D}y/2)$.
\end{remark}
 
The following concept was introduced by Pandey \cite{Pandey2022} as a full generalization of spherical $T$-designs (see Miezaki \cite{Miezaki}).

\begin{definition} [{Ellipsoidal $T$-design}] 
Let $X \subset C_D(r)$ be a finite non-empty subset and $T \subset \mathbb{N}$ be a (possibly infinite) subset.
A finite non-empty subset $X \subset C_D(r)$ is an {\it ellipsoidal $T$-design} if (\ref{EnD}) holds for all $f(x,y) \in H_{D,j}^{\mathbb{R}}[x,y]$ with $j \in T$.
\end{definition}

\begin{definition} \label{def:shell1}
We define the {\it norm $r$ shells} in $C_D(r)$ by
\begin{align*} 
\Lambda_D^r = C_D(r) \cap \mathbb{Z}^2.
\end{align*}
\end{definition}

\begin{theorem} [{\cite[Theorem 1.2]{Pandey2022}, Ellipsoidal shell}] \label{TDD}
If $D \in \{1,2,3,7,11,19,43,67,163 \}$, then every non-empty shells $\Lambda_D^r$ are ellipsoidal $T_D$-designs, where
\begin{align*}
T_D:= \begin{cases}
\mathbb{Z}^{+} \setminus 4\mathbb{Z}^{+} & (D=1),\\
\mathbb{Z}^{+} \setminus 6\mathbb{Z}^{+} & (D=3),\\
\mathbb{Z}^{+} \setminus 2\mathbb{Z}^{+} & (\mbox{otherwise}).
\end{cases}
\end{align*}
\end{theorem}

\begin{example} [{$5$-design}]
\label{exam:ES} 
The minimal shell of the Eisenstein lattice
\[
\Lambda_3^1 = \{ (\pm 1,0), (0,\pm1), (1,-1), (-1,1) \} 
\]
is an ellipsoidal $T$-design for $T=\mathbb{Z}^{+} \setminus 6\mathbb{Z}^{+}$.
In particular, it is an ellipsoidal $5$-design.
\end{example}

\section{Sobolev's theorem for ellipsoidal designs} \label{Sobolev}

Throughout this section, let $Q$ be a positive definite matrix defined by (\ref{Q}).
Let $G \subset \mathrm{O}(Q)$ be a finite subgroup and
\[
 \mathcal{P}_n(\mathbb{R}^2) = \{f(x,y) \in \mathbb{R}[x,y] \mid \deg f \leq n \}.
\]
We define the action of $M \in G$ on $f$ by
\[
 f^M(x,y) = f((x,y) \cdot {}^t(M^{-1})).
\]
We say that a polynomial $f$ is {\it $G$-invariant} if $f^M=f$ for all $M \in G$.
We denote the set of $G$-invariant polynomials in $\mathcal{P}_n(\mathbb{R}^2)$ by $\mathcal{P}_n(\mathbb{R}^2)^G$.

The following theorem is a classical result due to Sobolev:

\begin{theorem} [{\cite{Sobolev1962}, Sobolev's Theorem}] \label{Sobolevd=2}
Let $G \subset \mathrm{O}(2)$ be a finite subgroup and $X \subset S^1$ be a $G$-invariant finite subset.
Then the following conditions are equivalent:
\begin{enumerate}
\item[(i)] $X$ is a spherical $n$-design on $S^1$. 
\item[(ii)] Eq.\ $(\ref{EtD})$ for $D=1$ holds for all polynomials $f \in \mathcal{P}_n^G(\mathbb{R}^2)$. 
\end{enumerate}
\end{theorem}

The following is a main result in this section, which is a generalization of Sobolev's theorem for $S^1$.

\begin{theorem}[{Generalized Sobolev Theorem}] \label{ellSobolev}
Let $G \subset \mathrm{O}(Q)$ be a finite subgroup and $X \subset C_D(r)$ be a $G$-invariant finite subset.
Then the following conditions are equivalent:
\begin{enumerate}
\item[(i)] $X$ is an ellipsoidal $n$-design. 
\item[(ii)] Eq.\ $(\ref{EtD})$ holds for all polynomials $f \in \mathcal{P}_n^G(\mathbb{R}^2)$.
\end{enumerate}
\end{theorem}

\begin{proof} 
Since (i) $\Rightarrow$ (ii) is clear, it is sufficient to prove (ii) $\Rightarrow$ (i).
Let $D \equiv 1, \; 2 \pmod{4}$.
Suppose that (ii) holds.
Then for all $f \in \mathcal{P}_n(\mathbb{R}^2)$, note that $(1/\#G)\sum_{M \in G} f^M(x,y)$ is $G$-invariant.
Therefore, we have
\begin{align*}
\frac{1}{2 \pi \sqrt{D}} \int_{C_D(r)} f(x,y) d \tau(x,y)
& = \frac{1}{\#G} \sum_{M \in G} \frac{1}{2 \pi \sqrt{D}} \int_{C_D(r)} f(x,y) d \tau(x,y) \\ 
&= \frac{1}{2 \pi \sqrt{D}} \int_{C_D(r)} \left(\frac{1}{\#G} \sum_{M \in G} f^{M}(x,y)\right) d \tau(x,y)\\
&= \frac{1}{\#X} \sum_{(x,y) \in X} \left(\frac{1}{\#G} \sum_{M \in G} f^{M}(x,y) \right)\\
&= \frac{1}{\#X} \sum_{M \in G} \left( \frac{1}{\#G} \sum_{(x,y) \in X} f^{M}(x,y) \right)\\
&= \frac{1}{\#X} \sum_{M \in G} \left( \frac{1}{\#G} \sum_{(x,y) \in X} f(x,y) \right)\\
&=\frac{1}{\#X} \sum_{(x,y) \in X} f(x,y).
\end{align*}
The second equality and the fifth equality follow by the $G$-invariance of the measure $d \tau$ and $X$ respectively. 
The third equality holds since $X$ is an ellipsoidal $n$-design.
This completes the proof.
The case $D \equiv 3 \pmod{4}$ is similar.
\end{proof} 

By using \cref{ellSobolev}, we can prove the following theorem:

\begin{theorem} \label{tight2m+1}
Let $C^{(Q)}_{2m+1}$ be the cyclic group generated by $A_{2 \pi/(2m+1)}$ in \cref{OQ-inv}, and $X = (x,y)^{C^{(Q)}_{2m+1}}$ be the orbit of a point $(x,y) \in C_D(r)$ by $C^{(Q)}_{2m+1}$.
Then $X$ is an ellipsoidal $2m$-design. 
\end{theorem}

To prove this, we determine the basis of $\mathcal{P}_{2m}(\mathbb{R}^2)^G$.

\begin{lemma} \label{G-inv}
In the above setting, let $G = C^{(Q)}_{2m+1}$. Then
\begin{align*}
\mathcal{P}_{2m}(\mathbb{R}^2)^G=
\begin{cases}
 \langle (x^2+Dy^2)^j \mid 0 \leq j \leq m  \rangle_{\mathbb{R}} & (D \equiv 1, \; 2 \pmod{4}), \\
\left\langle \left(x^2+xy+\frac{1+D}{4}y^2 \right)^j \mid 0 \leq j \leq m  \right\rangle_{\mathbb{R}} & (D \equiv 3\pmod{4}).
\end{cases}
\end{align*}
\end{lemma}

\begin{proof} 
Suppose that $D \equiv 1, \; 2 \pmod{4}$. 
Let
\begin{align*}
B_{2k\pi/(2m+1)} =
\begin{pmatrix}
\cos \left(\frac{2k \pi}{2m+1} \right)& - \sin \left(\frac{2k \pi}{2m+1}\right)\\
\sin \left(\frac{2k \pi}{2m+1}\right) & \cos \left(\frac{2k \pi}{2m+1}\right)
\end{pmatrix}, &&
P =
\begin{pmatrix}
1 & 0\\
0 & \sqrt{D}
\end{pmatrix},
\end{align*}
and change variables as
\begin{align*}
\begin{pmatrix}
x' \\
y'
\end{pmatrix}
= P
\begin{pmatrix}
x \\
y
\end{pmatrix}.
\end{align*}
Let $H = C^{(E_2)}_{2m+1}$. Since $A_{2k\pi/(2m+1)}=P^{-1} B_{2k\pi/(2m+1)} P$ by the proof of \cref{OQ-inv}, we have
\begin{align*}
\mathcal{P}_{2m}(\mathbb{R}^2)^G
&=\{f \in \mathcal{P}_{2m} (\mathbb{R}^2) \mid f^{P^{-1} B_{2k\pi/(2m+1)} P} (x,y)=f(x,y)  \; \mbox{for all $k$}\}\\
&=\{f \in \mathcal{P}_{2m} (\mathbb{R}^2) \mid f^{B_{2k\pi/(2m+1)}} (x',y')=f(x',y')  \; \mbox{for all $k$}\}\\
&=\mathcal{P}_{2m}(\mathbb{R}^2)^H\\
&=\langle (x'^2+y'^2)^j \mid 0 \leq j \leq m  \rangle_{\mathbb{R}}\\
&=\langle (x^2+Dy^2)^j \mid 0 \leq j \leq m  \rangle_{\mathbb{R}}. 
\end{align*}
This completes the proof.
The case $D \equiv 3 \pmod{4}$ is similar. 
\end{proof}

\begin{proof} [{Proof of \cref{tight2m+1}}]
Suppose that $D \equiv 1, \; 2 \pmod{4}$. 
Let $G = C^{(Q)}_{2m+1}$.
Since $X$ is $G$-invariant, by \cref{ellSobolev}, it suffices to prove
\begin{align} \label{inv}
 \frac{1}{2 \pi \sqrt{D}} \int_{C_D(r)} f(x,y) d \tau(x,y)=\frac{1}{2m+1} \sum_{(x,y) \in X} f(x,y)
 \end{align}
for all $f \in \mathcal{P}_{2m}(\mathbb{R}^2)^G$. 
By \cref{G-inv}, for all $j$, the both sides of (\ref{inv}) for $f(x,y)=(x^2+Dy^2)^j $ are $r^j$.
This completes the proof.
The case $D \equiv 3 \pmod{4}$ is similar. 
\end{proof}

\section{Stroud-type bound and rational tight designs} \label{Stroud Bound}

In this section, we present an arithmetic proof of the Stroud-type bound for ellipsoidal designs (see Pandey \cite[p.\ 1247]{Pandey2022}), which is an extension of the Stroud-type bound for classical spherical designs (\cite[p.\ 364]{DGS1977}). 
Moreover, we establish a classification theorem for tight designs,
and then prove that a rational tight design on an ellipse $C_D(r)$ exists only if $D \in \{1,3 \}$.

\begin{theorem} [{\cite[p.\ 1247 Remark (2)]{Pandey2022}, Stroud-type bound}] \label{ESB}
If there exists an ellipsoidal $n$-design with $m$ points, then $m \geq n+1$.
\end{theorem}

By using the invariance of the measure $d\tau$ on $C_D(r)$ and the following inequality for the PTE problem,  we give an arithmetic proof.

\begin{theorem} [{PTE bound}] \label{PTEineq}
\hangindent\leftmargini
\textup{(i)}\hskip\labelsep 
$($Cf.\ \cite[Chapter 11, Theorem 1]{Borwein2002}, \cite[Theorem 4.1.7]{Ghiglione}$)$
If
 \[
[a_1, \ldots, a_m ] =_n  [b_1, \ldots, b_m]
 \]
is a solution of $\PTE_1$, then $m \geq n+1$.
\begin{enumerate}
\setcounter{enumi}{1}
\item[(ii)] $($\cite[Theorem 4.1.9]{Ghiglione}$)$ 
 If
 \[
[(a_{11}, \ldots a_{1r}), \ldots, (a_{m1}, \ldots a_{mr})] =_n [(b_{11}, \ldots b_{1r}), \ldots, (b_{m1}, \ldots b_{mr})]
 \]
is a solution of the $r$-dimensional PTE problem $($see \cref{prob:AT1}$)$, then $m \geq n+1$.
\end{enumerate}
\end{theorem}

\begin{proof} [{Proof of \cref{ESB}}] 
Let $\#X=m$ and
\[
 \frac{1}{m}\sum_{(x,y) \in X} f(x,y)= \int_{C_D(r)} f(x,y) d \tau(x,y)
\]
be an ellipsoidal $n$-design.
 Then by the invariance of the measure, we have
\[
\frac{1}{m} \sum_{(x',y') \in M^{-1}X} f(x',y') = \int_{C_D(r)} f(x,y) d \tau(x,y)
\]
for all $M \in \mathrm{O}(Q)$. 
Since $\mathrm{O}(Q)$ is an infinite group acting transitively on $C_D(r)$, we can take $M \in \mathrm{O}(Q)$ such that
$X$ and $M^{-1}X$ are disjoint if necessary. Therefore,
\[
\sum_{(x,y) \in X} f(x,y) = \sum_{(x',y') \in M^{-1}X} f(x',y')
\]
is a solution of $\PTE_2$. 
Therefore, we obtain $m \geq n+1$ by \cref{PTEineq} (ii). 
\end{proof}

\begin{definition} [{Tight desgin}] \label{def:tight1}
An ellipsoidal $n$-design with $n+1$ points is called a {\it tight $n$-design}.
\end{definition}

Next, we prove a classification theorem for tight ellipsoidal $n$-designs;
the proof of \cref{tight} (i) is similar to that of the main theorem of \cite{Hong}.

\begin{theorem} \label{tight} 
The following hold:
\begin{enumerate}
\item[(i)]  Let $X \subset C_D(r)$ be a tight $n$-design. 
Then there exists some $(x,y) \in C_D(r)$ such that $X=(x,y)^{C^{(Q)}_{2m+1}}$.
\item[(ii)]  A rational tight $n$-design on an ellipse $C_D(r)$ exists only if $D \in \{1,3 \}$. 
\end{enumerate}
\end{theorem}

First, we prove \cref{tight} (i). In doing so, we prepare some auxiliary lemmas.
Let
\[
(x',y') =
\begin{cases}
(x,\sqrt{D}y) & (D \equiv 1, \; 2 \pmod{4}), \\
(x+y/2,\sqrt{D}y/2) & (D \equiv 3 \pmod{4}).
\end{cases}
\]
Then $C_D(r)$ and $\{ z \in \mathbb{C} \mid |z| = \sqrt{r} \}$ are bijective by
\begin{align} \label{bij}
(x,y) \mapsto x'+y'\sqrt{-1}. 
\end{align}
With this identification, we obtain the following:

\begin{lemma} [{Cf.\ \cite[Lemma 1]{Hong}}] \label{elleq} 
Let $X = \{ \xi_1 , \ldots, \xi_m \} \subset C_D(r)$.
Then the following are equivalent.
\begin{enumerate}
\item[(i)] $X$ is an ellipsoidal $n$-design.
\item[(ii)] For all $1 \leq k \leq n$, it holds that
\[\sum_{i=1}^m \xi_i^k=0. \]
\item[(iii)] For all $1 \leq k \leq n$, it holds that
\[\sum_{1 \leq i_1 <i_2<\cdots < i_k\leq m} \xi_{i_1} \xi_{i_2} \cdots \xi_{i_k}=0. \]
\end{enumerate}
\end{lemma}

\begin{proof}
By \cref{HDj}, (i) and (ii) are equivalent.
By Newton's identity, (ii) and (iii) are equivalent.
\end{proof}
For $1 \leq k \leq m$, we define $\xi_k^* = \xi_k e^{-\sqrt{-1} \theta}$, where $\theta$ is chosen so that
\begin{align*}
\sqrt{r}^m e^{\sqrt{-1} m\theta}= \xi_1 \cdots \xi_m \;\; (0 \leq m \theta< 2 \pi).
\end{align*}

\begin{remark} \label{1to1} 
The multiplication by $e^{\sqrt{-1} \theta}$ in $\mathbb{C}$ corresponds to the action of matrix $A_{\theta}$ in \cref{OQ-inv}.
\begin{enumerate}
\item[(i)] Suppose that $D \equiv 1, \; 2 \pmod{4}$. Let $(x_k,y_k) \in C_D(r)$ and $\xi_k^* = x_k+y_k\sqrt{-D}$.
Then
\begin{align*}
A_{\theta} \begin{pmatrix}
x_k \\
y_k 
\end{pmatrix}
=\begin{pmatrix}
x_k \cos \theta -y_k\sqrt{D} \sin \theta\\
x_k \frac{\sin \theta}{\sqrt{D}}+y_k \cos \theta
\end{pmatrix}
\end{align*}
corresponds to $\xi_k^*e^{\sqrt{-1} \theta}=\xi_k$.

\item[(ii)] Suppose that $D \equiv 3 \pmod{4}$. Let $(x_k,y_k) \in C_D(r)$ and $\xi_k^*
=x_k+y_k(1+\sqrt{-D})/2$.
Then
\begin{align*}
A_{\phi}
\begin{pmatrix}
x_k\\
y_k
\end{pmatrix}
=\begin{pmatrix}
x_k \left(\cos \theta-\frac{\sin \theta}{\sqrt{D}} \right) -y_k\frac{(D+1)}{2\sqrt{D}} \sin \theta\\
x_k \frac{2\sin \theta}{\sqrt{D}} +y_k \left(\cos \phi+\frac{\sin \phi}{\sqrt{D}} \right) 
\end{pmatrix}
\end{align*}
corresponds to $\xi_k^*e^{\sqrt{-1} \theta}=\xi_k$.
\end{enumerate}
\end{remark}

In the following setting, we define 
\[g(z)=\prod_{k=1}^m (z-\xi_k^*) \]
and write
\[g(z)=\sum_{k=0}^m (-1)^k a_{m-k}z^{m-k}.\]
Then by \cref{elleq}, $X$ is an ellipsoidal $n$-design if and only if $a_i=0$ for all $m-n \leq i \leq m-1$.

\begin{lemma}  [{Cf.\ \cite[Lemma 2]{Hong}}]  \label{ai} 
The following hold:
\begin{enumerate} 
\item[(i)] $a_0=\sqrt{r}^m$.
\item[(ii)] $a_{m-k}=r^{-m/2+k} \overline{a_k}$ for all $0 \leq k \leq m$.
\item[(iii)] The following conditions are equivalent:
\begin{enumerate}
\item $X$ is an ellipsoidal $n$-design.
\item $a_{m-1}=\cdots=a_{m-n}=0$.
\item $a_1=\cdots=a_n=0$.
\end{enumerate}
\end{enumerate}
\end{lemma}

\begin{proof}
\hangindent\leftmargini
\textup{(i)}\hskip\labelsep 
By definition, $a_0=\xi_1^* \cdots \xi_m^*=\xi_1 \cdots \xi_m e^{- \sqrt{-1}m\theta}=\sqrt{r}^m$.
\begin{enumerate}
\setcounter{enumi}{1}
\item[(ii)] By (i) and $|\xi_i^*|^2=|\xi_i|^2=r$, we have
\begin{align*}
a_{m-k} &=\sum_{1 \leq i_1<i_2< \cdots< i_k \leq m} \xi_{i_1}^* \xi_{i_2}^* \cdots  \xi_{i_k}^*\\
&=\sum_{1 \leq i_1<i_2< \cdots< i_k \leq m} 
\frac{\sqrt{r}^m \xi_{i_1}^* \xi_{i_2}^* \cdots  \xi_{i_k}^*}{\xi_1^* \xi_2^* \cdots  \xi_m^*}\\
&=\sum_{1 \leq j_1<j_2< \cdots< i_{m-k} \leq m} 
\frac{\sqrt{r}^m}{\xi_{j_1}^* \xi_{j_2}^* \cdots  \xi_{j_{m-k}}^*}\\
&= \sum_{1 \leq j_1<j_2< \cdots< i_{m-k} \leq m} 
\frac{\sqrt{r}^m}{r^{m-k}} \overline{\xi_{j_1}^* \xi_{j_2}^* \cdots  \xi_{j_{m-k}}^*}\\
&=r^{-\frac{m}{2}+k} \sum_{1 \leq j_1<j_2< \cdots< i_{m-k} \leq m} 
\overline{\xi_{j_1}^* \xi_{j_2}^* \cdots  \xi_{j_{m-k}}^*}\\
&=r^{-\frac{m}{2}+k} \overline{a_k}.
\end{align*}
\item[(iii)] This follows from (ii) and 
\cref{elleq} (i). 
\qedhere
\end{enumerate}
\end{proof}

Now we prove a stronger version of \cref{tight}. 

\begin{proof}[{Proof of \cref{tight} $(i)$}]   
Let $X=\{\xi_1 , \ldots, \xi_m \} \subset C_D(r)$ be an $n$-design with $m \leq 2n+1$.
 Then we prove that $X \subset (x,y)^{C^{(Q)}_{2m+1}}$ for some $(x,y) \in C_D(r)$.
 
With $g$ defined above, we have
 \[
     g(z)=(z-\xi_1^*) \cdots (z-\xi_m^*)=z^m+(-1)^m \sqrt{r}^m
\]
 since $a_0=\sqrt{r}^m$ and $a_1=\cdots=a_{m-1}=0$ by \cref{ai}.
 
Suppose that $m$ is odd. Since $g(\xi_k^*)=0$, we have $\xi_k^{*m}=\sqrt{r}^m$.
Therefore, $\xi_k^*=\sqrt{r} e^{2l \pi \sqrt{-1}/m}$ for some $l$.
Thus, $\xi_k=\sqrt{r} e^{(2l \pi/m +\theta) \sqrt{-1}}$.
Hence $\xi_k$ is a vertex of the regular $m$-gon inscribed in the circle $\{ z \in \mathbb{C} \mid |z|=r \}$.
By pulling it back to $C_D(r)$ via (\ref{bij}), we obtain the assertion.
 
Suppose that $m$ is even.
Then $\xi_k^{*m}+\sqrt{r}^m=0$.
Therefore, $\xi_k^*=\sqrt{r}e^{(2l+1)\pi \sqrt{-1}/m}$ for some $l$.
Hence $\xi_k=\xi_k^* e^{\sqrt{-1} \theta}$ is a vertex of the regular $m$-gon inscribed in the circle $\{ z \in \mathbb{C} \mid |z|=r \}$. 
By pulling it back to $C_D(r)$ via (\ref{bij}), we obtain the assertion. 
\end{proof}

\begin{example} \label{RD}

Pandey \cite[p.\ 1247 Remark]{Pandey2022} remarks that there is a one-to-one correspondence between spherical $n$-designs and ellipsoidal $n$-designs. 
However, the rationality of design is not preserved under the linear transformation $P=\begin{pmatrix} 
1 & \frac{1}{2}\\
0 & \frac{\sqrt{3}}{2}
\end{pmatrix}$
 since $P \not\in M_2(\mathbb{Q})$. For example,
 the tight $5$-design
\[ X:=\left\{\pm \left(\frac{2t+1}{t^2+t+1},\frac{(t^2-1)}{t^2+t+1}  \right),
 \pm \left(\frac{(t^2-1)}{t^2+t+1},-\frac{t(t+2)}{t^2+t+1} \right),
 \pm \left(-\frac{t(t+2)}{t^2+t+1},\frac{2t+1}{t^2+t+1} \right)
 \right\} \]
in \cref{Dstr} corresponds to the regular hexagon
\small
\begin{align*}
H:=\left\{\pm \left(\frac{t^2+4t+1}{2(t^2+t+1)},\frac{\sqrt{3}(t^2-1)}{2(t^2+t+1)}  \right),
 \pm \left(\frac{t^2-2t-2}{2(t^2+t+1)},-\frac{\sqrt{3}t(t+2)}{2(t^2+t+1)} \right),
 \pm \left(\frac{-2t^2-2t+1}{2(t^2+t+1)},\frac{\sqrt{3}(2t+1)}{2(t^2+t+1)} \right)
 \right\}
  \end{align*}
 \normalsize
as $X \cdot {}^t P =H$. 
Therefore, to find rational designs, it is more natural to consider ellipsoidal designs rather than spherical designs.
See Section~\ref{PTEprob} for details.
\end{example}

We shall now move to the proof of \cref{tight} (ii).

\begin{lemma} \label{ratOQ}
If $X$ is a rational tight design, then we have 
 $A_{2\pi/(n+1)} \in M_2(\mathbb{Q})$.
\end{lemma}

\begin{proof}
Suppose that $X \subset C_D(r)$ is a rational tight $n$-design.
Then, by \cref{tight}, there exists $(x,y) \in X \cap \mathbb{Q}^2$ such that
$(x,y)^{A_{2\pi/(n+1)}} \in \mathbb{Q}^2$.

\bigskip

\begin{enumerate}
\item[(i)]
Suppose that $D \equiv 1, \; 2 \pmod{4}$.
Since
\[A_{2\pi/(n+1)}=\begin{pmatrix}
\cos  \left(2\pi/(n+1)\right) & -\sqrt{D} \sin \left(2\pi/(n+1)\right)\\
\frac{\sin \left(2\pi/(n+1)\right) }{\sqrt{D}}& \cos\left( 2\pi/(n+1)\right) 
\end{pmatrix}\]
by \cref{OQ-inv}, we can write
\begin{align*}
A_{2\pi/(n+1)} \begin{pmatrix}
x\\
y
\end{pmatrix}= 
 \begin{pmatrix}
x & -y\\
y & \frac{x}{D}
\end{pmatrix}
\begin{pmatrix}
\cos \left(2\pi/(n+1)\right)  \\
\sqrt{D} \sin \left(2\pi/(n+1)\right) 
\end{pmatrix}.
\end{align*}
This completes the proof since $(x,y) \in \mathbb{Q}^2$ and
\[\det \begin{pmatrix}
x & -y\\
y & \frac{x}{D}
\end{pmatrix}
=\frac{x^2+Dy^2}{D}=\frac{r}{D} \in \mathbb{Q}^{\times}. \]

\bigskip

\item[(ii)]
Suppose that $D \equiv 3 \pmod{4}$.
Since
\[
A_{2\pi/(n+1)}=\begin{pmatrix}
\cos \left(2\pi/(n+1)\right) -\frac{\sin \left(2\pi/(n+1)\right) }{\sqrt{D}} & -\frac{(D+1)}{2\sqrt{D}} \sin \left(2\pi/(n+1)\right) \\
\frac{2\sin \left(2\pi/(n+1)\right)}{\sqrt{D}}& \cos \left(2\pi/(n+1)\right)+\frac{\sin \left(2\pi/(n+1)\right)}{\sqrt{D}}
\end{pmatrix}\]
by \cref{OQ-inv}, we can write
\begin{align*}
A_{2\pi/(n+1)} \begin{pmatrix}
x\\
y
\end{pmatrix}= 
 \begin{pmatrix}
x & -\frac{2x+(D+1)y}{2D}\\
y & \frac{2x+y}{D}
\end{pmatrix}
\begin{pmatrix}
\cos \left(2\pi/(n+1)\right) \\
\sqrt{D} \sin \left(2\pi/(n+1)\right)
\end{pmatrix}.
\end{align*}
This completes the proof since $(x,y) \in \mathbb{Q}^2$ and
\[\det  \begin{pmatrix}
x & -\frac{2x+(D+1)y}{2D}\\
y & \frac{2x+y}{D}
\end{pmatrix}=\frac{2 \left(x^2+xy+\frac{1+D}{4}y^2 \right)}{D}=\frac{2r}{D} \in \mathbb{Q}^{\times}.
\qedhere
\]
\end{enumerate}
\end{proof}

As implied by the following lemma, if $\theta$ is a rational multiple of $2 \pi$, then the values of $\cos \theta$ and $\sin \theta$ are restricted.

\begin{lemma} [{\cite[Corollary 3.12]{Niven}}] \label{Niven}
If $\theta \in 2 \pi \mathbb{Q}$ and $\sin \theta \in \mathbb{Q}$ $($resp.\ $\cos \theta \in \mathbb{Q})$, then it holds that 
$\sin \theta \in \{0, \pm1, \pm 1/2 \}$ $($resp.\ $\cos \theta \in \{0, \pm1, \pm 1/2 \})$.
\end{lemma}

We are ready to complete the proof of \cref{tight} (ii). 
\begin{proof} [{Proof of \cref{tight} $(ii)$}] 
Suppose that $X \subset C_D(r)$ is
a rational tight $n$-design.
Then, by \cref{ratOQ},  there exists 
$t \in \mathbb{Q}$ such that 
\[\sin 2\pi/(n+1)=t\sqrt{D}, \;\; \cos  2\pi/(n+1)=\pm \sqrt{1-t^2D} \in \mathbb{Q}.\]
By \cref{Niven}, $\sqrt{1-t^2D}=0$ or $1$ or $1/2$.

\begin{enumerate} 
\item[(i)] If $ \sqrt{1-t^2D}=0$, then we have 
$t= \pm 1/\sqrt{D} \in \mathbb{Q}$. Since $D$ is a square-free integer, we must have $D=1$.
\item[(ii)] If $ \sqrt{1-t^2D}=1$, then 
we have $t=0$ by $D \in \mathbb{Z}_{>0}$. Therefore, 
we must have $\sin 2\pi/(n+1)=0$ and $\cos 2\pi/(n+1)=1$, which 
contradicts $n \in \mathbb{N}$.
\item[(iii)] If $ \sqrt{1-t^2D}=1/2$, then we have $t=\pm \sqrt{3}/(2\sqrt{D})$.
Since $D$ is a square-free integer, we must have $D=3$.
\qedhere
\end{enumerate}
\end{proof}


\section{Alpers--Tijdeman solution, Borwein solution and ellipsoidal designs} \label{PTEprob} 

In this section, we first provide a combinatorial criterion for the construction of solutions of $\PTE_2$ from a pair of ellipsoidal designs (\cref{DtoPTE}), and 
then, as a main theorem, obtain a parametric ideal solution of degree $5$ (\cref{Dstr}).
Moreover, we establish a two-dimensional extension of Borwein's solution (\ref{Borwein}) (\cref{2DBorwein}). 
Another main theorem of this section is the equivalence between the two-dimensional Borwein solution and the Alpers--Tijdeman solution (\ref{ATsol}) over $\mathbb{Q}$ (\cref{2BvsAT}). As an important corollary of this result, we discover a family of ellipsoidal $5$-designs among the Alpers--Tijdeman solution.

\begin{proposition} \label{DtoPTE}
If there exists a pair of ellipsoidal $n$-designs with $m$ points, then so does a solution of degree $n$ and size $m$ for $\PTE_2$.
\end{proposition}
\begin{proof}
Let 
\begin{align*}
\frac{1}{m}\sum_{(x,y) \in X} f(x,y)= \int_{C_D(r)} f(x,y)  d \tau(x,y), 
&& \frac{1}{m}\sum_{(x',y') \in Y} f(x',y')= \int_{C_D(r)} f(x,y) d \tau(x,y)
\end{align*}
be a pair of ellipsoidal $n$-designs. 
Since $\mathrm{O}(Q)$ is an infinite group acting transitively on $C_D(r)$, we can take $M \in \mathrm{O}(Q)$ such that
$X$ and $M^{-1}Y$ are disjoint if necessary. Since
$d\tau$ is $\mathrm{O}(Q)$-invariant,
\begin{align*}
\frac{1}{m}\sum_{(x'',y'') \in M^{-1}Y} f(x'',y'')= \int_{C_D(r)} f(x,y) d \tau(x,y)
\end{align*}
is also an ellipsoidal $n$-design.
Therefore, we have
\begin{align*}
\sum_{(x,y) \in X} f(x,y)=\sum_{(x'',y'') \in M^{-1}Y} f(x'',y''),
\end{align*}
which yields a solution of $\PTE_2$.
Since $\mathrm{O}(Q)$ is an infinite group acting transitively on $C_D(r)$, we can also obtain a parametric solution of $\PTE_2$ by taking various $M$.
\end{proof}

\begin{remark} \label{rem:lattice1}
By \cref{TDD}, the minimal shell $\Lambda^1_3$ of the Eisenstein lattice is a $5$-design.
Thus, by combining with \cref{OQ-inv}, we obtain a parametric solution of $\PTE_2$. 
It was Mishima et al.\ \cite{MLSU} who first observed a one-dimensional version of \cref{DtoPTE}.
\end{remark}

\begin{definition} [{Rational design/solution}] \label{def:rational}
An ellipsoidal design is said to be {\it rational} if the coordinates of points are all rational numbers.
Similarly, we define a {\it rational solution} to the PTE problem. 
\end{definition}

\begin{remark} \label{rat}
For $D=1$, it is not known whether there exists a rational $4$-design (see Cui et al.\ \cite[\S 1.3]{CXX}).
They also mention that finding spherical designs $X$ such that the inner products between points in $X$ are all rational is tantamount to finding rational designs on some ellipsoid.  
Since $S^1$ and $C_D(r)$ are homeomorphic, there may not seem to be much difference between spherical design and ellipsoidal design.
However, by considering rational points, we can find an essential gap between the two notions, as will be seen in \cref{Dstr}.
\end{remark}

As opposed to the suggestion by Cui et al., infinitely many rational designs for $D \neq 1$ are included in the following parametric solution of $\PTE_2$ (see also \cref{RD}):

\begin{theorem} \label{Dstr}
There exists an infinite family of rational ellipsoidal $5$-design on the ellipse $C_3(1)$ given by
\begin{align} \label{RTED}
 \left\{\pm \left(\frac{2t+1}{t^2+t+1},\frac{(t^2-1)}{t^2+t+1}  \right),
 \pm \left(\frac{(t^2-1)}{t^2+t+1},-\frac{t(t+2)}{t^2+t+1} \right),
 \pm \left(-\frac{t(t+2)}{t^2+t+1},\frac{2t+1}{t^2+t+1} \right)
 \right\}. 
 \end{align}
In particular, there exists a parametric rational ideal solution of $\PTE_2$ given by
\begin{multline} \label{2dimPTE}
 \left[\pm \left(\frac{2t_1+1}{t_1^2+t_1+1},\frac{(t_1^2-1)}{t_1^2+t_1+1}  \right),
 \pm \left(\frac{(t_1^2-1)}{t_1^2+t_1+1},-\frac{t_1(t_1+2)}{t_1^2+t_1+1} \right),
 \pm \left(-\frac{t_1(t_1+2)}{t_1^2+t_1+1},\frac{2t_1+1}{t_1^2+t_1+1} \right)
 \right]\\
 =_5 
 \left[\pm \left(\frac{2t_2+1}{t_2^2+t_2+1},\frac{(t_2^2-1)}{t_2^2+t_2+1}  \right),
 \pm \left(\frac{(t_2^2-1)}{t_2^2+t_2+1},-\frac{t_2(t_2+2)}{t_2^2+t_2+1} \right),
 \pm \left(-\frac{t_2(t_2+2)}{t_2^2+t_2+1},\frac{2t_2+1}{t_2^2+t_2+1} \right)
 \right],
 \end{multline}
where $t_1$, $t_2 \in \mathbb{Q}$.
\end{theorem}

\begin{proof}
For any rational number $t$,
 the intersection of the line $y=t(x-1)$ and the ellipse $C_3(1)$ are $(1,0)$ and $((t^2-1)/(t^2+t+1),-t(t+2)/(t^2+t+1))$.
The transformation that sends $(0,-1)$ to $((t^2-1)/(t^2+t+1),-t(t+2)/(t^2+t+1))$ is represented by $A_{\phi}$ (see \cref{OQ-inv}) with 
\[(\cos \phi, \sin \phi)=\left(\frac{t^2+4t+1}{2(t^2+t+1)},\frac{\sqrt{3}(t^2-1)}{2(t^2+t+1)} \right). \]
Note that $\Lambda^1_3=\{(\pm1, 0), (0,\pm1), (1,-1), (-1,1) \}$ is a $5$-design by \cref{TDD} and
$d\tau$ is invariant under
\[A_{\phi}=\begin{pmatrix}
\frac{2t+1}{t^2+t+1} & -\frac{(t^2-1)}{t^2+t+1}\\
\frac{(t^2-1)}{t^2+t+1} & \frac{t(t+2)}{t^2+t+1}
\end{pmatrix} \] by \cref{OQ-inv}. 
Therefore, $\Lambda^1_3 \cdot {}^{t}A_{\phi}$ is also a rational design for all $t \in \mathbb{Q}$.
By combining this with \cref{DtoPTE}, we obtain the desired rational ideal solutions of $\PTE_2$.
\end{proof}

\bigskip

Now, the following is another main theorem of this section:

\begin{theorem} \label{2BvsAT} 
The Alpers--Tijdeman solution $(\ref{ATsol})$  is equivalent to the two-dimensional Borwein solution $(\ref{2dimBorwein})$ over $\mathbb{Q}$. 
In particular, both solutions contain our parametric solution $(\ref{2dimPTE})$ over $\mathbb{Q}$.
\end{theorem}

Before proving \cref{2BvsAT}, we give a brief introduction to the $r$-dimensional PTE problem.

\begin{problem} [{\cite{Alpers-Tijdeman}}] \label{prob:AT1}
The {\it $r$-dimensional PTE problem $(\PTE_r)$ of degree $n$ and size $m$}, 
asks whether there exists a disjoint pair of multisets
\[
\{(a_{11}, \ldots, a_{1r}), \ldots, (a_{m1}, \ldots, a_{mr})\}, \{(b_{11}, \ldots, b_{1r}), \ldots, (b_{m1}, \ldots, b_{mr})\} \subset \mathbb{Z}^r
\]
such that
\[
\sum_{i=1}^m a_{i1}^{k_1} \cdots a_{ir}^{k_r}=\sum_{i=1}^m b_{i1}^{k_1} \cdots b_{ir}^{k_r} 
\quad (1 \leq k_1+ \cdots +k_r \leq n).
\]
\end{problem}

\begin{definition} \label{equiv}
\hangindent\leftmargini
\textup{(i)}\hskip\labelsep  (Cf.\ \cite[p.\ 87]{Borwein2002}, \cite[p.\ 405]{Alpers-Tijdeman})
A solution
$[\mathbf{a}_1, \ldots, \mathbf{a}_m]=_n [\mathbf{b}_1, \ldots, \mathbf{b}_m]$ of $\PTE_r$ {\it contains} a solution $[\mathbf{c}_1, \ldots, \mathbf{c}_m] =_n [\mathbf{d}_1, \ldots, \mathbf{d}_m]$ over $\mathbb{Q}$ if there exist $M \in M_r(\mathbb{Q})$ and $\mathbf{e} = (e_1,\ldots,e_r)$ such that 
\begin{align*}
 \{ \mathbf{a}_i M + \mathbf{e} \mid  1 \leq i \leq m \} =\{\mathbf{c}_i\mid  1 \leq i \leq m \}, &&
 \{ \mathbf{b}_i  M  + \mathbf{e} \mid 1 \leq i \leq m \} = \{\mathbf{d}_i \mid 1 \leq i \leq m\}
\end{align*}
as multisets.
Two solutions of $\PTE_r$ are {\it equivalent} if they contain each other, i.e., map each other by an affine transformation.

\begin{enumerate}
\setcounter{enumi}{1}
\item[(ii)] (Cf.\ \cite[p.\ 86]{Borwein2002}) A solution $[\mathbf{a}_1, \ldots, \mathbf{a}_m] =_n [\mathbf{b}_1, \ldots, \mathbf{b}_m]$ of $\PTE_r$ is {\it symmetric} if 
\begin{align*}
\{\mathbf{a}_i  \mid 1 \leq i \leq m \}=\{-\mathbf{a}_i  \mid i = 1 \leq i \leq m\}, && \{\mathbf{b}_i  \mid 1 \leq i \leq m \}=\{-\mathbf{b}_i  \mid 1 \leq i \leq m\}
\end{align*}
 as multisets.

\item[(iii)]  (\cite[p.\ 26]{MLSU})
A solution $[\mathbf{a}_1, \ldots, \mathbf{a}_m] =_n [\mathbf{b}_1, \ldots, \mathbf{b}_m]$ of $\PTE_r$ is {\it linear} if there exists $S \subset \{1,\ldots,m\}$ such that $\sum_{i \in S} \mathbf{a}_i=\sum_{i \in S} \mathbf{b}_i=0$. 
\end{enumerate}
\end{definition}

\begin{remark} \label{EDsymlin}
Our solution $(\ref{2dimPTE})$ of $\PTE_2$ is symmetric, which is moreover linear in the sense that the six points on each side can be divided into two triples,
each of whose sum equals $0$. 
What is the most remarkable in the arguments below is that the symmetry and the linearity of (\ref{2dimPTE}) (and (\ref{2dimBorwein})) are effectively used to reduce the  candidates of simultaneous equations that appear in the proof of \cref{2BvsAT,vsBorwein}. 
\end{remark}

We now give some auxiliary lemmas used in the proof of \cref{2BvsAT}.
Hereafter, we think of points $\mathbf{a}_i, \mathbf{b}_i, \mathbf{c}_i, \mathbf{d}_i \in \mathbb{Q}^r$ 
as column vectors rather than row vectors, and so multiply $M \in M_r(\mathbb{Q})$ from the left.
The following lemma is straightforward, which is, however, useful for further arguments below. 

\begin{lemma} \label{KeyLem}
Assume that a symmetric solution
$[\pm \mathbf{a}_1, \ldots, \pm \mathbf{a}_k] =_n [\pm \mathbf{b}_1, \ldots, \pm \mathbf{b}_k]$
of $\PTE_r$ contains a symmetric one
$[\pm \mathbf{c}_1, \ldots, \pm \mathbf{c}_k] =_n [\pm \mathbf{d}_1, \ldots, \pm \mathbf{d}_k]$.
Then it holds that $\mathbf{e}=\mathbf{0}$.
Moreover, assume the linearity condition that
\begin{align} \label{A-1}
\sum_{i=1}^k \mathbf{a}_i=\sum_{i=1}^k \mathbf{b}_i=\sum_{i=1}^k \mathbf{c}_i=\sum_{i=1}^k \mathbf{d}_i= \mathbf{0}. 
\end{align}
We also let
\begin{align} \label{A-2}
M \mathbf{a}_i=\mathbf{c}_i, && M \mathbf{b}_i=\mathbf{d}_i
\end{align}
for $1 \leq i \leq k-1$. Then it holds that 
\begin{align*}
M \mathbf{a}_k=\mathbf{c}_k, && M \mathbf{b}_k=\mathbf{d}_k.
\end{align*}
\end{lemma}

\begin{proof}
By the symmetry of the solutions, we have
\begin{align*}
\sum_{i=1}^k \mathbf{a}_i+ \sum_{i=1}^k (-\mathbf{a}_i)=\sum_{i=1}^k \mathbf{c}_i+\sum_{i=1}^k (-\mathbf{c}_i)=
\mathbf{0}, 
\end{align*}
which implies the former statement. 
The latter statement follows from the assumptions (\ref{A-1}) and (\ref{A-2}), i.e.,
\begin{align*}
M \mathbf{a}_k=M \left(\sum_{i=1}^k \mathbf{a}_i -\sum_{i=1}^{k-1}  \mathbf{a}_i \right) = M \cdot \mathbf{0} - \sum_{i=1}^{k-1} \mathbf{c}_i=\mathbf{c}_k, \\
M  \mathbf{b}_k=M \left(\sum_{i=1}^k \mathbf{b}_i -\sum_{i=1}^{k-1}  \mathbf{b}_i \right) = M \cdot \mathbf{0} - \sum_{i=1}^{k-1} \mathbf{d}_i=\mathbf{d}_k. 
\end{align*} 
These complete the proof. 
\end{proof}

In what follows, we restrict our attention to solutions of degree $5$ for $\PTE_r$ for $r=1, 2$. 

The next lemma not only serves as a tool for the proof of \cref{2BvsAT},
but also gives an idea for ^^ cyclically extending' a solution of $\PTE_1$ to that of $\PTE_2$; we will return to the discussion on such ^^ lifting constructions' in a forthcoming paper.

\begin{lemma} [Cyclic lifting] \label{lifting} 
Assume that
$c_i = Aa_i$
and $d_i=Ab_i$ for $1 \leq i \leq 3$, and
\begin{align} 
\label{solB1-2}
[\pm (a_1,a_2), \pm (a_2,a_3), \pm (a_3,a_1)] =_5 [\pm (b_1,b_2), \pm (b_2,b_3), \pm (b_3,b_1)] ,\\
\label{solB2-2}
[\pm (c_1,c_2), \pm (c_2,c_3), \pm (c_3,c_1)] =_5 [\pm (d_1,d_2), \pm (d_2,d_3), \pm (d_3,d_1)] 
\end{align}
are solutions of $\PTE_2$.
Then the solution $(\ref{solB1-2})$ contains the solution $(\ref{solB2-2})$.
\end{lemma}

\begin{proof}
By the assumption, we have
\begin{align*}
\begin{pmatrix}
c_i\\
c_{i+1}
\end{pmatrix}=\begin{pmatrix}
A & 0\\
0 & A
\end{pmatrix}
\begin{pmatrix}
a_i\\
a_{i+1}
\end{pmatrix}, &&
\begin{pmatrix}
d_i\\
d_{i+1}
\end{pmatrix}=\begin{pmatrix}
A & 0\\
0 & A
\end{pmatrix}
\begin{pmatrix}
b_i\\
b_{i+1}
\end{pmatrix}
\end{align*}
for $1 \leq i \leq 3$, where the subscripts of $a_i,b_i,c_i,d_i$ are reduced cyclically modulo $3$. The proof is thus completed.
\end{proof}

The following auxiliary lemma gives a two-dimensional extension of Borwein's solution (\ref{Borwein}) and is of theoretical interest in its own right.  

\begin{proposition} [{Two-dimensional Borwein solution}] \label{2DBorwein} 
There exists a parametric rational ideal solution of $\PTE_2$ given by
\begin{multline}  \label{2dimBorwein}
[\pm (2(n+m),-nm-n-m+3), \pm (-nm-n-m+3,nm-n-m-3),\pm (nm-n-m-3, 2(n+m)) ]\\
=_5  [\pm (2(n-m),nm-n+m+3), \pm (nm-n+m+3,-nm-n+m-3),\pm (-nm-n+m-3, 2(n-m))],
\end{multline}
which is, as with $(\ref{2dimPTE})$ $($see \cref{EDsymlin}$)$, a symmetric solution that consists of two pairs of three linear points on each side.
\end{proposition}
\begin{proof}
We can check it by direct calculation.
Indeed, let $x$ and $y$ denote the first and the second coordinates, respectively. Then the sums of $x^k$ or $y^k$ for $0 \leq k \leq 5$ coincide by (\ref{Borwein}). 
By the symmetry of (\ref{2dimBorwein}), the sums of odd monomials are $0$. 
The sums of $xy$, $x^3y$, $x^2y^2$ and $xy^3$ are $-2(m^2+3)(n^2+3)$, $-2(m^2+3)^2(n^2+3)^2$, $2(m^2+3)^2(n^2+3)^2$ and $-2(m^2+3)^2(n^2+3)^2$ respectively. 
\end{proof}

We now come to a key lemma (\cref{vsBorwein}) in the proof of \cref{2BvsAT}. 
We shall recall the Alpers--Tijdeman solution (\ref{ATsol}), namely
\begin{multline*}
 [ (0,0),(2a + b,b),(3a + b,3a + 3b),(2a,6a + 4b),(-b,6a + 3b),(-a-b,3a + b)]\\
 =_5 [(2a,0),(3a + b,3a + b),(2a + b,6a + 3b),(0,6a + 4b),(-a-b,3a + 3b),(-b,b)],
\end{multline*} 
and denote by $\{(\alpha_i,\beta_i) \}_{i=1}^6$ the first six points of the Alpers--Tijdeman solution (\ref{ATsol}). 
Suppose that (\ref{ATsol}) contains (\ref{2dimPTE}) over $\mathbb{Q}$.
Since (\ref{ATsol}) is symmetric, there exist
\begin{align*}
M:=\begin{pmatrix}
A & B\\
C & D
\end{pmatrix} \in M_2(\mathbb{Q}), && \mathbf{e}:=
\begin{pmatrix}
E\\
F
\end{pmatrix} \in \mathbb{Q}^2
\end{align*}
such that 
\begin{align*}
\begin{pmatrix}
6(A+3B)a+12Bb\\
6(C+3D)a+12Db\
\end{pmatrix}+ 6\begin{pmatrix}
E\\
F
\end{pmatrix} = \sum_{i=1}^6 M \begin{pmatrix}
 \alpha_i\\
 \beta_i\\
 \end{pmatrix} + 6 \mathbf{e}=\mathbf{0},
\end{align*}
i.e.,
\begin{align*}
\begin{pmatrix}
E\\
F
\end{pmatrix} =
\begin{pmatrix}
-(A+3B)a-2Bb\\
-(C+3D)a-2Db\
\end{pmatrix}.
\end{align*}

With the above affine transformation, the solution $(\ref{ATsol})$ is equivalent to
\begin{align} \label{ATsol2} 
\begin{split}
&\left[\pm (-(A+3B)a-2Bb,-(C+3D)a-2Db), \right. \\
&\left.  \pm((-A+3B)a+(-A+B)b,(-C+3D)a+(-C+D)b), \pm (2Aa+(A+B)b, 2Ca+(C+D)b) \right]\\
=_5&\left[\pm ((A-3B)a-2Bb, (C-3D)a-2Db), \right. \\
&\left.  \pm(-2Aa+(-A+B)b, -2Ca+(-C+D)b), \pm ((A+3B)a+(A+B)b,(C+3D)a+(C+D)b) \right],
\end{split}
\end{align}
which is, as with (\ref{2dimPTE}) (recall \cref{EDsymlin}), 
a symmetric solution that consists of two pairs of three linear points on each side.
We also note that, with the translation by $(-a,-3a-2b)$,
the solution (\ref{ATsol}) is equivalent to
\begin{multline} \label{ATsol-2} 
[\pm (a, 3a+2b), \pm (a+b,-3a-b), \pm (-2a-b,-b)]\\
=_5 [\pm (a, -3a-2b), \pm (a+b,3a+b), \pm (-2a-b,b)],
\end{multline}
which is again a symmetric solution that consists of two pairs of three linear points on each side.

\begin{proposition} \label{vsBorwein}
The two-dimensional Borwein solution $(\ref{2dimBorwein})$ contains our parametric solution $(\ref{2dimPTE})$ over $\mathbb{Q}$. 
\end{proposition}

\begin{proof} 
By \cref{lifting}, the statement can be reduced to showing that the one-dimensional Borwein solution (\ref{Borwein}) contains the parametric solution
\begin{multline} \label{1dimPTE}
\left[\pm \frac{(2t_1+1)}{t_1^2+t_1+1}, \pm \frac{(t_1^2-1)}{t_1^2+t_1+1}, 
\pm \frac{t_1(t_1+2)}{t_1^2+t_1+1} \right] \\
 =_5
\left[\pm \frac{(2t_2+1)}{t_2^2+t_2+1}, \pm \frac{(t_2^2-1)}{t_2^2+t_2+1}, 
\pm \frac{t_2(t_2+2)}{t_2^2+t_2+1} \right] 
\end{multline}
over $\mathbb{Q}$.

Let
\begin{equation} \label{eq:parameter1}
\begin{gathered}
x_1 = 2(n+m), \;\; x_2=-nm-n-m+3, \;\; x_3=nm-n-m-3,\\
y_1 = 2(n-m), \;\; y_2 = nm-n+m+3, \;\; y_3=-nm-n+m-3,\\
a_1 = \frac{2s+1}{s^2+s+1}, \;\; a_2 = \frac{s^2-1}{s^2+s+1}, \;\; a_3= \frac{-s(s+2)}{s^2+s+1},\\
b_1 = \frac{2t+1}{t^2+t+1}, \;\; b_2 = \frac{t^2-1}{t^2+t+1}, \;\; b_3 = \frac{-t(t+2)}{t^2+t+1}.
\end{gathered}
\end{equation}
It is obvious that the solutions (\ref{Borwein}) and (\ref{1dimPTE}) are symmetric. Also, since the two-dimensional solution (\ref{2dimPTE}) (resp.\ (\ref{2dimBorwein}))
satisfies the linearity condition (\ref{A-1}), by restricting all the points to the first components, we find that the one-dimensional solution (\ref{1dimPTE}) (resp.\ (\ref{Borwein})) also satisfies the linearity condition (\ref{A-1}).
Then by \cref{KeyLem}, it suffices to prove that for all $s$, $t \in \mathbb{Q}$, there exists $A \in \mathbb{Q}^{\times}$ such that
\begin{align} \label{simeq}
\begin{cases}
Ax_1 &= a_1,\\
Ax_2 &= a_2,\\
Ay_1 &= b_1,\\
Ay_2 &= b_2, 
\end{cases}
\end{align}
 since both $Ax_3=a_3$ and $Ay_3=b_3$ follow from the linearity, and both $A(-x_i)=-a_i$ and $A(-y_i)=-b_i$ for $1 \leq i \leq 3$ follow from the symmetry.

What we do below is to directly solve (\ref{simeq}) over $\mathbb{Q}$.
If $A = 0$, then we have $s = -1/2$ by the first equation, but $s = \pm 1$ by the second equation, which is a contradiction.

Thus we may consider the case $A \ne 0$.
Then the first and the third equations of (\ref{simeq}) imply that
\begin{align} \label{eq:parameter2}
n = \frac{a_1+b_1}{4A}, && \quad m = \frac{a_1-b_1}{4A}.
\end{align} 
Substituting them into the second equation of (\ref{simeq}), we have
\begin{align} \label{a2}
a_ 2=A(-nm-n-m+3) =
-\frac{a_1^2 - b_1^2}{16A}-\frac{a_1}{2}+3A,
\end{align}
i.e.,
\[
48A^2-8(a_1+2a_2)A-(a_1^2-b_1^2)=0,
 \]
which implies that
\[
A = \frac{(a_1+2a_2) \pm \sqrt{4a_1^2+4a_1a_2+4a_2^2-3b_1^2}}{12}.
\]
Note that
\begin{align} \label{a1+2a2}
a_1+2a_2=\frac{2s^2+2s-1}{s^2+s+1},
\end{align}
and
\begin{align*}
4a_1^2+4a_1a_2+4a_2^2-3b_1^2 &=
4-3\frac{(2t+1)^2}{(t^2+t+1)^2} =
\Big( \frac{2t^2+2t-1}{t^2+t+1} \Big)^2
\end{align*}
since $a^2+a_1a_2+a_2^2=1$ for $(a_1,a_2) \in C_3(1)$.
For example, when
\begin{align} \label{A}
A =\frac{1}{12} \left(\frac{2s^2+2s-1}{s^2+s+1}+\frac{2t^2+2t-1}{t^2+t+1} \right) 
=\frac{(2st+s+t-1)(2st+s+t+2)}{12(s^2+s+1)(t^2+t+1)},
\end{align}
it follows from (\ref{eq:parameter1}) that
\begin{align*} 
a_1+b_1 =\frac{(s+t+1)(2st+s+t+2)}{(s^2+s+1)(t^2+t+1)}, &&
a_1-b_1 = \frac{(t-s)(2st+s+t-1)}{(s^2+s+1)(t^2+t+1)},
\end{align*}
which imply, with (\ref{eq:parameter2}) and (\ref{A}), that 
\begin{align*}
n =\frac{a_1+b_1}{4A}=\frac{3(s+t+1)}{2st+s+t-1}, &&
m =\frac{a_1-b_1}{4A}=\frac{3(t-s)}{2st+s+t+2}.
\end{align*}
Now, it follows from (\ref{a2}) and (\ref{eq:parameter2}) that
\begin{align*} 
a_2 = -A \cdot \frac{a_1+b_1}{4A} \cdot \frac{a_1-b_1}{4A} - \frac{a_1}{2} + 3A = -Anm - \frac{a_1}{2}+3A.
\end{align*}
By combining this with the third equation of (\ref{simeq}), we have
\begin{align} \label{Ay2}
Ay_2 = Anm - \frac{A}{2} \cdot 2(n-m) + 3A = -\frac{a_1}{2}-a_2 -\frac{b_1}{2}+6A. 
\end{align}
Finally, by substituting (\ref{eq:parameter1}), (\ref{a1+2a2}) and (\ref{A}) into (\ref{Ay2}), we obtain
\begin{align*} 
Ay_2
& = \frac{1}{2}\left(-a_1-2a_2 -\frac{(2t+1)}{t^2+t+1}+\left(a_1+2a_2+\frac{2t^2+2t-1}{t^2+t+1}\right)\right),
\end{align*}
which coincides with $b_2$.
\end{proof}

We are now in a position to complete the proof of \cref{2BvsAT}.

\begin{proof} [{Proof of \cref{2BvsAT}}]
We first prove that the modified Alpers--Tijdeman solution (\ref{ATsol2}) contains the two-dimensional Borwein solution
\begin{multline}  \label{2dB1}
[\pm (2(n+m),-nm-n-m+3), \pm (-nm-n-m+3,nm-n-m-3),\pm (nm-n-m-3, 2(n+m)) ]\\
=_5  [\pm (2(n-m),nm-n+m+3), \pm (nm-n+m+3,-nm-n+m-3),\pm (-nm-n+m-3, 2(n-m)) ].
\end{multline}
By the linearity of the solutions and \cref{KeyLem}, it is sufficient to show that,
for all $m, n \in \mathbb{Q}$, there exist $A$, $B$, $C$, $D$, $a$, $b \in \mathbb{Q}$ such that
\begin{align} \label{claim}
\begin{split}
\begin{cases} 
\pm (-(A+3B)a-2Bb,-(C+3D)a-2Db) &=\pm (2(n+m),-nm-n-m+3),\\
\pm ((-A+3B)a+(-A+B)b,(-C+3D)a+(-C+D)b) &=\pm (-nm-n-m+3,nm-n-m-3),\\
\pm ((A-3B)a-2Bb, (C-3D)a-2Db) &=\pm (2(n-m),nm-n+m+3),\\
\pm (-2Aa+(-A+B)b, -2Ca+(-C+D)b) &=\pm (nm-n+m+3,-nm-n+m-3). 
\end{cases}
\end{split}
\end{align}
By the symmetry of the solutions, it suffices to consider (\ref{claim}) only for positive signs as follows:
\begin{align} \label{simeq-2} 
\begin{cases}
(-A - 3B)a-2Bb & = 2(n+m),\\
(-A + 3B)a+(-A+B)b & = -nm-n-m+3,\\
(A - 3B)a-2Bb & = 2(n-m),\\
-2Aa+(-A+B)b & = nm-n+m+3,\\
(-C -3D)a-2Db & = -nm-n-m+3,\\
(-C + 3D)a+(-C+D)b & = nm-n-m-3,\\
(C - 3D)a-2Db  & = nm-n+m+3, \\
-2Ca+(-C+D)b &  = -nm-n+m-3.
\end{cases}
\end{align}

We may assume that $m \neq 0$ since otherwise the two multisets in (\ref{2dB1}) coincide. 
What we do below is to solve (\ref{simeq-2}) for $a=1$; the argument below works for other values of $a$, but we here let $a=1$ just for simplicity.
Then the first and the third equations of (\ref{simeq-2}) imply that
\begin{align} \label{eq1-9}
A=-2m, && b=-\frac{n}{B}-\frac{3}{2}.
\end{align}
Substituting (\ref{eq1-9}) into the second equation of (\ref{simeq-2}), we have
\begin{align*}
3B^2+(2mn-6)B-4mn=0.
\end{align*}
Thus, we have $B=-2mn/3$ or $B=2$.
For example, when $B=-2mn/3$,
we have $b=-3(m-1)/(2m)$. 
Then the fifth and the seventh equations of (\ref{simeq-2}) imply that
\begin{align*}
C=m(n+1), && D=\frac{m(n-3)}{3}.
\end{align*}
Concerning the fourth, the sixth and the eighth equations of (\ref{simeq-2}), the right-hand sides follow by substituting the values of $A$, $B$, $C$, $D$, $a$ and $b$ into the left-hand sides.
Therefore, the Alpers--Tijdeman solution (\ref{ATsol}) (or equivalently (\ref{ATsol2}))  
contains the two-dimensional Borwein solution (\ref{2dB1}).

Conversely, we prove that the solution (\ref{2dB1}) contains an equivalent form (\ref{ATsol-2}) of (\ref{ATsol}).
By the linearity of the solutions and \cref{KeyLem}, it is sufficient to show that,
for all $m, n \in \mathbb{Q}$, there exist $A$, $B$, $C$, $D$, $a$, $b \in \mathbb{Q}$ such that
\begin{multline} \label{claim-1}
\begin{split}
\begin{cases} 
\pm (2A(n+m)+B(-nm-n-m+3),2C(n+m)+D(-nm-n-m+3))
\\ \quad
=\pm (a, 3a+2b),\\
\pm (A(-nm-n-m+3)+B(nm-n-m-3),C(-nm-n-m+3)+D(nm-n-m-3))\\ 
\quad
=\pm (a+b,-3a-b),\\
\pm (2A(n-m)+B(nm-n+m+3),2C(n-m)+D(nm-n+m+3))
\\ \quad
=\pm (a, -3a-2b),\\
\pm (A(nm-n+m+3)+B(-nm-n+m-3),C(nm-n+m+3)+D(-nm-n+m-3))\\
\quad
=\pm (-2a-b,b).
\end{cases}
\end{split}
\end{multline}
By the symmetry of the solutions, it suffices to consider (\ref{claim-1}) only for positive signs as follows:
\begin{align} \label{simeq-3} 
\begin{cases}
2A(n+m)+B(-nm-n-m+3)&=a,\\
A(-nm-n-m+3)+B(nm-n-m-3) &=a+b,\\
2A(n-m)+B(nm-n+m+3) & =a,\\
A(nm-n+m+3)+B(-nm-n+m-3) &=-2a-b,\\
2C(n+m)+D(-nm-n-m+3) &=3a+2b,\\
C(-nm-n-m+3)+D(nm-n-m-3) &=-3a-b,\\
2C(n-m)+D(nm-n+m+3) &=-3a-2b,\\
C(nm-n+m+3)+D(-nm-n+m-3) &=b.
\end{cases}
\end{align}

We may assume that
$0 \notin \{a, 3a+2b\}$, since otherwise the two multisets in (\ref{ATsol-2}) coincide. 
What we do below is to solve (\ref{simeq-3}) for $n=0$; the argument below works for other values of $n$, but we here let $n=0$ just for simplicity.
Then the first and the third equations of (\ref{simeq-3}) imply that 
\begin{align} \label{eq1-10}
A= \frac{a}{6}, && B=\frac{a}{3}.
\end{align}
Substituting (\ref{eq1-10}) into the second equation of (\ref{simeq-3}), we have
\begin{align*}
m=-\frac{3a+2b}{a}.
\end{align*}
Then the fifth and the seventh equations of (\ref{simeq-3}) imply that
\begin{align*}
C=- \frac{a}{2}, && D=0.
\end{align*}
Concerning the fourth, the sixth and the eighth equations of (\ref{simeq-3}), the right-hand sides follow by substituting the values of $A$, $B$, $C$, $D$, $m$ and $n$ into the left-hand sides.
Therefore, the two-dimensional Borwein solution (\ref{2dB1}) contains the Alpers--Tijdeman solution (\ref{ATsol}).

Finally, by \cref{vsBorwein}, we obtain the latter statement, which completes the proof of the theorem.
\qedhere
\end{proof} 

\begin{remark} \label{2dimto1dim}
There is a solution in Borwein's parametric solution that is not contained in our parametric solution (\ref{1dimPTE}). 
Indeed, if we substitute $(m,n)=(2,4)$ in Borwein's solution, then we obtain 
\[ [\pm1, \pm 11, \pm 12] =_5 [\pm 4, \pm 9, \pm 13] \]
 of \cite[p.\ 87]{Borwein2002}. 
 This solution is not equivalent to our solution (\ref{1dimPTE}) over $\mathbb{Q}$.
Indeed, suppose contrary. Then there exists $A \in \mathbb{Q}$ such that
\begin{align*}
532A^2 &=2 A^2 (1^2+11^2+12^2)\\
&= 2\left( \left(\frac{2t+1}{t^2+t+1} \right)^2
+ \left(-\frac{(t^2-1)}{t^2+t+1} \right)^2
+ \left(\frac{t(t+2)}{t^2+t+1} \right)^2\right)=4.
\end{align*}
Therefore, $A^2=1/133$, which contradicts $A \in \mathbb{Q}$.
\end{remark}

\section{Concluding remarks} \label{sect:conclusion}

The following fact ensures the existence of an ellipsoidal $n$-design with sufficiently many points:

\begin{theorem} [{\cite{Arias1988}}] \label{thm:Arias} 
Let $f_1,\ldots,f_l$ be continuous functions on a closed interval $[a,b]$. 
Then there exist $x_1,\ldots,x_N \in [a,b]$ such that
\[
\frac{1}{N} \sum_{i=1}^N f_k(x_i) = \frac{1}{b-a} \int_a^b f_k(t) 
dt
\ \text{for every $k=1,\ldots,l$}.
\]
\end{theorem}
Indeed, we can take $[a,b]=[0,2\pi]$ and 
\begin{align*}
f_{i,j}(t) :=
\begin{cases}
 (\sqrt{r} \cos t )^i ( \sqrt{r}\sin t/\sqrt{D})^j & (D \equiv 1, 2 \pmod{4}),\\
 (\sqrt{r}(\cos t-\sin t/\sqrt{D}) )^i ( 2\sqrt{r} \sin t/\sqrt{D} )^j  & (D \equiv 3 \pmod{4})
\end{cases}
\end{align*}
with $0 \leq i+j \leq n$ as $f_k$'s in our case.
For a refinement of \cite{Arias1988} including the asymptotic behavior of the number of points of spherical designs, we refer the readers to \cite[Theorem 1]{BRV}.

Cui et al. \cite[Theorem 1.3]{CXX} showed the existence of designs over $S^1$, 
where one component of each point is in $\mathbb{Q}(\sqrt{q} \mid q \; \mbox{is a prime})$.
Similar results can also be obtained for ellipsoidal designs, but we will not go into them in this paper.

\begin{center}
* * *
\end{center}

Mishima et al.\ \cite[Theorem 1.4 (i)]{MLSU}  
 classified all designs of degree $5$ with $6$ rational points for Chebyshev measure $(1-t^2)^{-1/2}dt/\pi$ and then obtained the following parametric ideal solutions for $\PTE_1$:
\begin{multline}  \label{MLSU}
  \left[\pm \frac{(2t_1^2-22t_1-13)}{14(t_1^2+t_1+1)}, \pm \frac{(-13t_1^2-4t_1+11)}{14(t_1^2+t_1+1)}, 
\pm \frac{(11t_1^2+26t_1+2)}{14(t_1^2+t_1+1)} \right]\\
 =_5 
  \left[\pm \frac{(2t_2^2-22t_2-13)}{14(t_2^2+t_2+1)}, \pm \frac{(-13t_2^2-4t_2+11)}{14(t_2^2+t_2+1)}, 
\pm \frac{(11t_2^2+26t_2+2)}{14(t_2^2+t_2+1)} \right] \; (t_1, \; t_2 \in \mathbb{Q}).
\end{multline} 
This ideal solution is not equivalent to our parametric solution (\ref{1dimPTE}) over $\mathbb{Q}$.
Suppose contrary. Then by considering 
sums of squares, there exists $A \in \mathbb{Q}$ such that
$3A^2 =4$, i.e., $A^2=4/3$, which is a contradiction. 

It is remarkable that the parametric solution (\ref{MLSU}) is obtained from $C_3(3/4)$, as seen from the following result:
\begin{proposition} \label{LMSU} 
Let 
\begin{align*}
x_1'
= \frac{2t^2-22t-13}{14(t^2+t+1)}, && x_2'
= \frac{-13t^2-4t+11}{14(t^2+t+1)}, 
&& x_3'
= \frac{11t^2+26t+2}{14(t^2+t+1)}.
\end{align*}
Then, 
\[\{\pm (x_1',x_2'), \pm (x_2',x_3'), \pm (x_3', x_1') \} \subset C_3(3/4) \]
is an ellipsoidal $5$-design.
\end{proposition}
\begin{proof}
Since
\begin{align*}
x_1'^2+x_1'x_2'+x_2'^2=x_2'^2+x_2'x_3'+x_3'^2=x_3'^2+x_3'x_1'+x_1'^2=\frac{3}{4},
\end{align*}
we obtain
\[\pm (x_1',x_2'), \; \pm (x_2',x_3'), \;  \pm (x_3', x_1') \in C_3(3/4). 
 \]
Since $\Lambda^{3/4}_3=\emptyset$, \cref{TDD} is not directly applicable, but the conditions on moments can be checked by direct calculation by using \cref{measure}. 
\end{proof}

By applying the ^^ product formula' (see \cite[Theorem 1.2.2]{Tchernychova}) to $5$-designs with $6$ rational points as in (\ref{MLSU}), we can construct a
$5$-design with $6^2$ rational points  for the equilibrium measure 
$\prod_{i=1}^2(1-x_i)^{-1/2} dx_1dx_2/\pi^2$ (see \cite{HL2025}), 
and thereby obtain a solution of $\PTE_2$ of degree $5$ and size $36$.
Whereas, the parametric solution (\ref{MLSU}) given by \cref{DtoPTE,LMSU} has only $6$ rational points, which is therefore ideal.

\cref{vsBorwein} states that Borwein's solution includes a design structure.
A natural question asks whether there exists a parametric ideal solution of degree $5$ for $\PTE_1$ except for Borwein's solution (\ref{Borwein}).

\begin{theorem} [{\cite[pp.\ 629--630]{Chernick} Chernick's solution}]
There exists a parametric ideal solution of $\PTE_1$ given by
\begin{multline*} 
[\pm (-5m^2+4mn-3n^2), \pm (-3m^2+6mn+5n^2), \pm (-m^2-10mn-n^2)] \\
=_5 [\pm (-5m^2+6mn+3n^2), \pm (-3m^2-4mn-5n^2), \pm (-m^2+10mn-n^2)].
\end{multline*} 
\end{theorem}

\begin{theorem} \label{thm:Chernick1}
Chernick's parametric solution is not equivalent to our parametric solution $(\ref{1dimPTE})$ over $\mathbb{Q}$.
 \end{theorem}

\begin{proof}
Suppose contrary. Then there exists $A \in \mathbb{Q}$ such that
\[
 14A^2(5m^2+2mn+n^2)(m^2-2mn+5n^2) =4
\] 
since
\begin{multline*} 
(-5m^2+4mn-3n^2)^2+(-3m^2+6mn+5n^2)^2+(-m^2-10mn-n^2)^2\\
= 7(5m^2+2mn+n^2)(m^2-2mn+5n^2).
\end{multline*} 
If $n=0$, then $70A^2m^4=4$, which contradicts $A$, $m \in \mathbb{Q}$.
Thus $n \neq 0$.
Let $x =m/n$, $y = 2/(An)$, and consider the smooth curve
\[
 C: y^2=14(5x^2+2x+1)(x^2-2x+5)
\]
of genus $1$.
Since $C(\mathbb{Q}_7)=\emptyset$, we conclude that $C(\mathbb{Q})=\emptyset$.

Indeed, suppose that there exists $(x,y) \in C(\mathbb{Q})$.
If $x \in \mathbb{Z}$, then by reducing $x$ modulo $7$, 
the $7$-adic valuation of the right-hand side is $1$, which is impossible.
Thus, $x \not\in \mathbb{Z}$. Write $x=u/7^n$ for some $n \in \mathbb{N}$ and $u \in \mathbb{Q}$ such that $\ord_7(u)=0$.
Then we have
\[
(7^{2n}y)^2=14(5u^2+2u \cdot 7^n+7^{2n})(u^2-2\cdot 7^n+5 \cdot 7^{2n}).
\]
Since $\ord_7(u)=0$, the $7$-adic valuation of the right-hand side is $1$, which is impossible.
This completes the proof.
\end{proof}

We can also check $C(\mathbb{Q}_7)=\emptyset$ by ``IsLocallySoluble" command of MAGMA \cite{Bosma-Cannon-Playoust}.

To study the interrelations between $\PTE_r$ and geometric designs is also left for a challenging future work.

\begin{center}
* * *
\end{center}

Coppersmith et al.\ \cite{CMSV} investigated the PTE problem over certain imaginary quadratic number fields,
and obtained ideal solutions of class number $1$ (\cref{h=1}).
They constructed ideal solutions with 
degree $9$ or $11$ over $\mathbb{Z}[\sqrt{-1}]$, 
degree $8$ over $\mathbb{Z}[\sqrt{-2}]$ and 
degree $8$ or $11$ over $\mathbb{Z}[\sqrt{-3}]$.
On the other hand, we can construct a curious parametric solution
over imaginary quadratic fields with class number possibly $\geq 2$.

In \cref{DtoPTE}, we have provided a construction of solutions of the PTE problem from a pair of ellipsoidal designs. We shall look at a construction of solutions in a number field of class number $h$ (see \cref{h=1}) from a different design structure.

\begin{proposition} \label{Chebyshev3-2}
There exists a parametric ideal solution of the PTE problem given by
\begin{align} \label{Chebyshev3-2}
 \left[z_1, \frac{-z_1-3 \pm \sqrt{-3z_1^2-6z_1-5}}{2} \right]
 =_2  \left[ z_2, \frac{-z_2-3 \pm \sqrt{-3z_2^2-6z_2-5}}{2} \right].
 \end{align}
\end{proposition} 
\begin{proof}
Consider the following complex integration:
\[a_k:=-\frac{1}{4 \pi \sqrt{-1}} \int_{S^1 \subset \mathbb{C}} z^k e^{-\frac{2}{z}} dz. \]
We determine $z_1$, $z_2$, $z_3$ such that
\begin{align*} 
\frac{1}{3} z_1^k+\frac{1}{3} z_2^k+\frac{1}{3} z_3^k=a_k, \quad k=1,2,
\end{align*}
which corresponds to a {\it $($weighted$)$ design for the Bessel weight}; the details will be explained later.
Since $a_k=(-2)^k/(k+1)!$
by the residue theorem, we have
\begin{align*}
z_1+z_2+z_3 &=3a_1=-3,\\
z_1^2+z_3^2+z_3^2&=3a_2=2.
\end{align*}
By the first equality, $z_3=-z_1-z_2-3$.
We substitute it into the second to obtain
\[2z_1^2+2z_2^2+2z_1z_2+6z_1+6z_2+7=0 .\]
Then we obtain
\[
z_1=\frac{-z_2-3 \pm \sqrt{-3z_2^2-6z_2-5}}{2},
\]
which provides a parametric solution of degree $2$ given by  (\ref{Chebyshev3-2}).
\end{proof}

By specializing $z_1$ and $z_2$, we can obtain solutions of the PTE problem of degree $2$ over imaginary quadratic fields with class number $\geq 2$. 

\begin{definition} [{PTE solution of class number $h$}] \label{h=1}
A solution of the one-dimensional PTE problem is called a {\it solution of class number $h$} if all the solutions are in a number field of class number $h$.
For the $r$-dimensional PTE problem, it means that all the components of solutions are
in a number field of class number $h$.
\end{definition}

\begin{example}
 Let $(z_1,z_2)=(-2,0)$ in (\ref{Chebyshev3-2}).
Then we obtain the following ideal solution of the PTE problem of class number $2$:
\begin{align*} 
\left[-2, \frac{-1 \pm \sqrt{-5}}{2}\right] =_2 
\left[0,  \frac{-3 \pm \sqrt{-5}}{2} \right].
\end{align*}
 \end{example}
 
Actually, the integration $a_k$ appeared in (\ref{Chebyshev3-2}) is the $k$-th moment of Bessel polynomials
\begin{align*}
y_n(z):=\sum_{k=0}^n \frac{(n+k)!}{(n-k)!k!} \left(\frac{z}{2} \right)^k.
\end{align*}
An ellipsoidal design for $D = 1$, i.e., a spherical design over $S^1$, is closely related to Gegenbauer polynomials (for example see \cite[Section 5]{SU20}).
As with Gegenbauer polynomials, Bessel polynomials are also in the class of {\it classical orthogonal polynomials} (see~\cite[\S2.1]{SU20} for the definition)
and are orthogonal with respect to the weight function 
\[
w(z)=-\frac{1}{4 \pi \sqrt{-1}} e^{-\frac{2}{z}},
\]
where the path of integration is $S^1 \subset \mathbb{C}$, i.e.,
\[
\int_{S^1} y_m(z)y_n(z)w(z)dz=0
\]
for $m \neq n$ (cf.\ \cite[p.\ 104]{KF}). 
In \cite{M_Bessel}, the first author proved the existence and non-existence of
`weighted' designs for the weight function of Bessel polynomials, as an extension of results of Sawa--Uchida \cite{SU20} for classical orthogonal polynomials such as Hermite, Legendre and Jacobi polynomials.

We close this paper by mentioning some open questions, to which we will intend to return in the near future.
\begin{problem}
\hangindent\leftmargini
\textup{(i)}\hskip\labelsep 
Does there exist an infinite family of designs of degree $\geq 3$ for Bessel polynomials? 
If this is the case, how large can be the class number of the corresponding parametric solution of $\PTE_1$.
\begin{enumerate}
\setcounter{enumi}{1}
\item[(ii)] Explore the connection between ellipsoidal designs and designs for Bessel polynomials.
\end{enumerate}
\end{problem}

\noindent {\bf Acknowledgements.}
The authors are grateful to Eiichi Bannai who kindly informed us \cite[p.\ 208, Problem 2]{BBIT_book} and then motivated us to study the existence, as well as
construction of rational ellipsoidal designs.
They also thank Yukihiro Uchida for carefully reading the earlier draft and giving helpful comments and suggestions on the presentation of the proof \cref{2BvsAT}.
Finally, they would like to thank Akihiro Munemasa, Hiroaki Narita, Koji Tasaka for giving us valuable informations on ellipsoidal harmonics and geometric designs,
and Kosuke Minami for fruitful discussions on the Stroud-bound for spherical designs.

\end{document}